\documentclass{amsart}
\usepackage{amsmath}
\usepackage{latexsym}
\usepackage{xcolor}
\usepackage{amscd}
\usepackage[all]{xy}
\usepackage{enumerate}
\usepackage{mathrsfs}
\usepackage{hyperref}
\hypersetup{
setpagesize=false,
 bookmarksnumbered=true,%
 bookmarksopen=true,%
 colorlinks=true,%
 linkcolor=blue,
 citecolor=red,
}
\usepackage{soul}
\usepackage{graphicx}
\usepackage{comment}
\newtheorem{thm}{Theorem}[section]
\newtheorem{lem}[thm]{Lemma}

\theoremstyle{definition}

\def\M{\mathcal{M}_{g}}
\def\H{\mathcal{H}_{g}}

\def\Sg{\Sigma_{g}}

\begin{document}

\title[mapping class groups]
{On generating mapping class groups by pseudo-Anosov elements}

\author[S. Hirose]{Susumu Hirose}
\address{Department of Mathematics, Faculty of Science and Technology, Tokyo University of Science, Noda, Chiba 278-8510, Japan}
\email{hirose-susumu@rs.tus.ac.jp}

\author[N. Monden]{Naoyuki Monden}
\address{Department of Mathematics, Faculty of Science, Okayama University, Okayama 700-8530, Japan}
\email{n-monden@okayama-u.ac.jp}

\begin{abstract}
Wajnryb proved that the mapping class group of a closed oriented surface is generated by two elements. 
We prove that the mapping class group is generated by two pseudo-Anosov elements. 
In particular, if the genus is greater than or equal to nine, we can take the generators to two conjugate pseudo-Anosov elements with arbitrarily large dilatations. 
Another result we prove is that the mapping class group is generated by two conjugate reducible but not periodic elements if the genus is greater than or equal to eight. 
We also give similar results to the first and third results for the hyperelliptic mapping class group when the genus is greater than or equal to one. 
\end{abstract}

\maketitle

\section{Introduction}
Let $\Sg$ be a closed oriented surface of genus $g$, and let $\M$ be the mapping class group of $\Sg$, i.e., the group of isotopy classes of orientation-preserving homeomorphisms of $\Sg$. 
The hyperelliptic mapping class group $\H$ is the subgroup of $\M$ consisting of isotopy classes of orientation preserving homeomorphisms on $\Sg$ that commute with some fixed hyperelliptic involution $\iota:\Sg \to \Sg$. 
We note that $\M = \H$ for $g=0,1,2$ and in particular that $\mathcal{M}_0 = \mathcal{H}_0 = \{1\}$.

The problem of finding generating sets for groups is a classical 
one and has been studied by many authors. 
The first finite generating set for $\M$ was given by Dehn \cite{De}, and the generating set consists of $2g(g-1)$ Dehn twists. 
After that, Lickorish \cite{Li} showed that $3g-1$ Dehn twists generate $\M$. 
Moreover, Humphries \cite{Hu} proved that $2g+1$ Dehn twists suffice to generate $\M$ and that $2g$ (or less) Dehn twists cannot generate $\M$. 
It is well-known that $\H$ is generated by $2g+1$ Dehn twists (see \cite{BH}). 
If we do not require the generators of $\M$ to be Dehn twists, then the number of generators can be reduced. 
In fact, Lu \cite{Lu} showed that $\M$ is generated by three elements. 
Here, since $\M$ is not cyclic, at least two elements are needed to generate $\M$. 
For this, Wajnryb \cite{Wa} gave a generating set for $\M$ consisting of two elements for $g\geq 1$, that is, the generating set is minimal. 
Since then, various minimal generating sets for $\M$ have been constructed (see, for example, \cite{Ko1, Yi2, BK}). 
The problem of finding torsion (or involution) generating sets has also been studied, for example, for finite simple groups. 
There have been many results on torsion (or involution) generating sets for $\M$ (see, for example, \cite{Mac, MP, Luo, BF, Ka, Ko1, Mo3, Du0, La, Ko2, Yi1, Yi2}). 
In particular, Korkmaz gave a minimal torsion generating set of $\M$ consisting of two torsion elements of order $4g+2$ for $g\geq 3$ \cite{Ko1} (which was improved to order $g$ by Yildiz \cite{Yi2} for $g\geq 6$). 
A minimal generating set for $\H$ consisting of two torsion elements of order $4g+2$ and $2g+2$ was constructed by Korkmaz \cite{Ko1} for $g=1,2$ and Stukow \cite{St} for $g\geq 3$. 
We note that $\H$ can not be generated by involutions for all $g\geq 1$ from homological 
reasons 
by the work of Birman and Hilden \cite{BH}.

By Nielsen-Thurston classification, every element $f$ in $\M$ (resp. $\H$) has a representative homeomorphism $\phi$ which is of one of the following three types: (1) \textit{periodic} i.e., $\phi^m = \mathrm{id}$ for some integer $m>0$, (2) \textit{reducible} i.e., $\phi$ leaves invariant a finite collection of pairwise disjoint essential simple closed curves on $\Sigma_g$, or (3) \textit{pseudo-Anosov} \cite{FM, Th}. 
An element $f$ in $\M$ is said to be \textit{periodic} (resp. \textit{reducible} or \textit{pseudo-Anosov}) if a representative of $f$ is periodic (resp. reducible or pseudo-Anosov).   
Periodic elements in $\M$ are just torsion elements. 
Note that there is an element in $\M$ which is periodic and reducible as shown in Figure~\ref{rotation} in the next section. 
A simple example of a reducible but not periodic element is a Dehn twist, and a torsion element in $\M$ of order $4g+2$ is periodic but not reducible. 
On the other hand, if $f$ is a pseudo-Anosov 
element, then it is neither periodic nor reducible.

The minimal torsion generating set in \cite{Ko1} (resp. \cite{Yi2}) consists of two periodic but not reducible elements (resp. two periodic and reducible elements). 
Our results are motivated by these works. 
We give minimal generating sets for $\H$ and $\M$ consisting of two pseudo-Anosov elements (see Theorem~\ref{thm:1}) and a minimal generating set for $\M$ consisting of two reducible but not periodic elements (see Theorem~\ref{thm:4}). 
\begin{thm}\label{thm:1}
For $g \geq 1$, the hyperelliptic mapping class group $\H$ and the mapping class group $\M$ are generated by two pseudo-Anosov elements. 
\end{thm}

We next focus on minimal normal generating sets for $\M$. 
Since any Dehn twists about nonseparating simple closed curves are conjugate, the above-mentioned generating sets given by Lickorish \cite{Li} and Humphries \cite{Hu} (resp. Birman-Hilden \cite{BH}) are normal generating sets of $\M$ (resp. $\H$). 
Moreover, the above-mentioned minimal torsion generating set in \cite{Ko1} (resp. \cite{Yi2}) consisting of two periodic but not reducible (resp. periodic and reducible) elements is also a normal one. 
With this background, Theorem~\ref{thm:3} (resp. \ref{thm:4}) below gives a minimal normal generating set for $\M$ consisting of two pseudo-Anosov (resp. reducible but not periodic) elements. 
In Theorem~\ref{thm:50} below, we give a minimal normal generating set for $\H$ consisting of two reducible but not periodic elements. 
However, we do not know whether $\H$ is generated by two conjugate periodic (resp. pseudo-Anosov) elements. 

Another motivation for Theorem~\ref{thm:3} comes from the work of Lanier and Margalit \cite{LM}. 
To explain this, we introduce some 
terminology for pseudo-Anosov elements. 
Let $g\geq 2$. 
When $f$ is a pseudo-Anosov element in $\M$, $f$ has a representative $\phi$ such that there are transverse measured foliations $(\mathcal{F}^s, \mu_s)$ and $(\mathcal{F}^u, \mu_u)$ on $\Sigma_g$, and a real number $\lambda>1$ so that $\phi \cdot (\mathcal{F}^s, \mu_s) = (\mathcal{F}^s, \lambda^{-1} \mu_s)$ and $\phi \cdot (\mathcal{F}^u, \mu_u) = (\mathcal{F}^u, \lambda \mu_u)$. 
Then, $\mathcal{F}^s$ and $\mathcal{F}^u$ are called the \textit{stable} and \textit{unstable foliations}, and $\lambda$ is called the \textit{dilatation} or \textit{stretch factor} of $\phi$. 
Note that two conjugate pseudo-Anosov elements have the same dilatation. 
With these terminologies, we can state Lanier and Margalit's result. 
They showed that $\M$ is normally generated by pseudo-Anosov elements with arbitrarily large dilatations for each $g\geq 1$ 
(see 
Theorem 1.3 and Proposition 4.2 in \cite{LM}). 
Theorem~\ref{thm:3} is a refinement of this result in the sense that two such conjugate pseudo-Anosov elements suffice to generate $\M$ if $g\geq 9$.

\begin{thm}\label{thm:3}
The mapping class group $\M$ is generated by two conjugate pseudo-Anosov elements with arbitrarily large dilatations if $g\geq 9$, and by three conjugate pseudo-Anosov elements with arbitrarily large dilatations if $g\geq 3$. 
\end{thm}
\begin{thm}\label{thm:4}
The mapping class group $\M$ is generated by two conjugate reducible but not periodic elements if $g\geq 8$, and by three conjugate reducible but not periodic elements if $g\geq 3$. 
\end{thm}
\begin{thm}\label{thm:50}
The hyperelliptic mapping class group $\H$ is generated by two conjugate reducible but not periodic elements if $g\geq 1$. 
\end{thm}
It is well-known that the mapping class group $\M$ of $\Sigma_g$ can be identified with the orbifold fundamental group of the moduli space $\mathbb{M}_g$ of Riemann surfaces of genus $g$. 
When we equiv $\mathbb{M}_g$ with the Teichm\"{u}ller metric, pseudo-Anosov elements in $\M$ correspond to geodesic loops in $\mathbb{M}_g$. 
In particular, the logarithm of the dilatation of a pseudo-Anosov element in $\M$ is the length of a geodesic loop in $\mathbb{M}_g$. 
Then, we can rewrite Theorem~\ref{thm:3} in terms of the geometry of normal covers of $\mathbb{M}_g$. 
Theorem~\ref{thm:3} implies that we can take two (resp. three) geodesic loops in $\mathbb{M}_g$ with the same arbitrary large length that generate the orbifold fundamental group of $\mathbb{M}_g$ if $g\geq 9$ (resp. $g\geq 3$).

\vspace{0.1in}
\noindent \textit{Acknowledgements.} 
The first author was supported by JSPS KAKENHI Grant Numbers JP20K03618. 
The second author was supported by JSPS KAKENHI Grant Numbers JP20K03613. 
The authors would like to thank Eiko Kin and Mustafa Korkmaz for their comments on an earlier version of this paper.

\section{Preliminaries}\label{Pre}
In this section, we recall the basic facts about the mapping class group and present examples of pseudo-Anosov elements and reducible but not periodic elements. 
More details can be found in \cite{FM}. 
Let $\M$ be the mapping class group of the closed oriented surface $\Sigma_g$ of genus $g$. 
In this paper, we use the same symbol for a diffeomorphism and its isotopy class, or a simple closed curve and its isotopy class. 
For a simple closed curve $c$ on $\Sigma_g$, we denote by $t_c$ the right-handed Dehn twist about $c$. 
For $f_1$ and $f_2$ in $\M$, the notation $f_2f_1$ means that we first apply $f_1$ and then $f_2$.

Let $a,b,c,d$ be simple closed curves on $\Sg$. 
Throughout this paper, we will use the following four relations in $\M$ repeatedly. 
\begin{itemize}
\item For any $f \in \M$, we have $t_{f(a)} = f t_a f^{-1}$.

\item If $a$ is disjoint from $b$, then $t_a t_b = t_b t_a$, $t_b(a) = a$ and $t_a(b) = b$.

\item If $a$ intersects $b$ transversely at exactly one point, then $t_bt_a(b) = a$, and hence $t_a(b) = t_b^{-1}(a)$. 
This gives the \textit{braid relation} $t_at_bt_a = t_bt_at_b$.

\item Let $x,y,z$ be the interior curves on a subsurface of genus $0$ with four boundary curves $a,b,c,d$ on $\Sg$ as in Figure~\ref{lanterncurves}. 
Then, the \textit{lantern relation} $t_d t_c t_b t_a = t_x t_y t_z$ holds. 
The lantern relation was discovered by Dehn \cite{De} and rediscovered by Johnson \cite{Jo1}. 
\end{itemize}

We assume that the surface of this paper is the $yz$-plane and that $\Sigma_g$ is embedded in $\mathbb{R}^3$ as in Figure~\ref{rotation} (resp. Figure~\ref{humphries1}) such that it is invariant under the rotation $r$ by $-\frac{2\pi}{g}$ about the $x$-axis (resp.  the rotation $\iota$ by $\pi$ about the $y$-axis, i.e., the hyperelliptic involution, and the rotation $r_\pi$ by $\pi$ about the $x$-axis). 
Note that $r_\pi$ is an element in $\H$ since $\iota r_\pi = r_\pi \iota$. 
\begin{figure}[hbt]
  \centering
       \includegraphics[height=3cm]{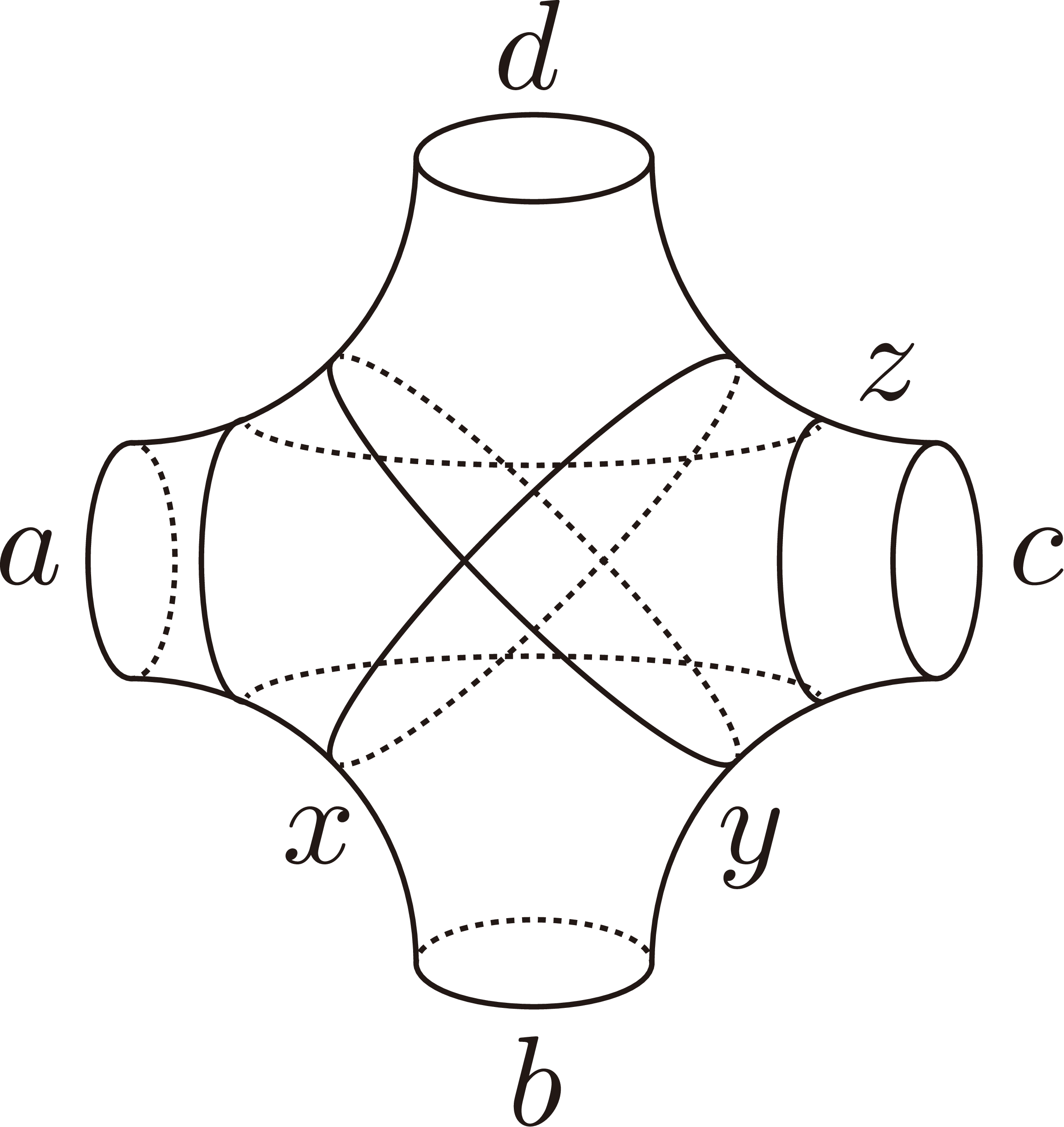}
       \caption{Simple closed curves $a,b,c,d,x,y,z$ on $\Sg$ for the lantern relation.}
       \label{lanterncurves}
  \end{figure}

The notations $a_i,b_i,c_i$ (resp. $\alpha_j$, $\beta$) will always denote the simple closed curves on $\Sigma_g$ shown in Figure~\ref{rotation} (resp.~\ref{humphries1}) for $i=1,2,\ldots,g$ (resp. $j=1,2,\ldots,2g+1$). 
We identify surfaces in Figure 2 and Figure 3 so that $\alpha_1 = a_1$, $\beta=a_2$, $\alpha_{2i} = b_i$, $\alpha_{2i+1} = c_{i}$ for $i=1,2,\ldots,g-1$. 
\begin{figure}[hbt]
  \centering
       \includegraphics[scale=.20]{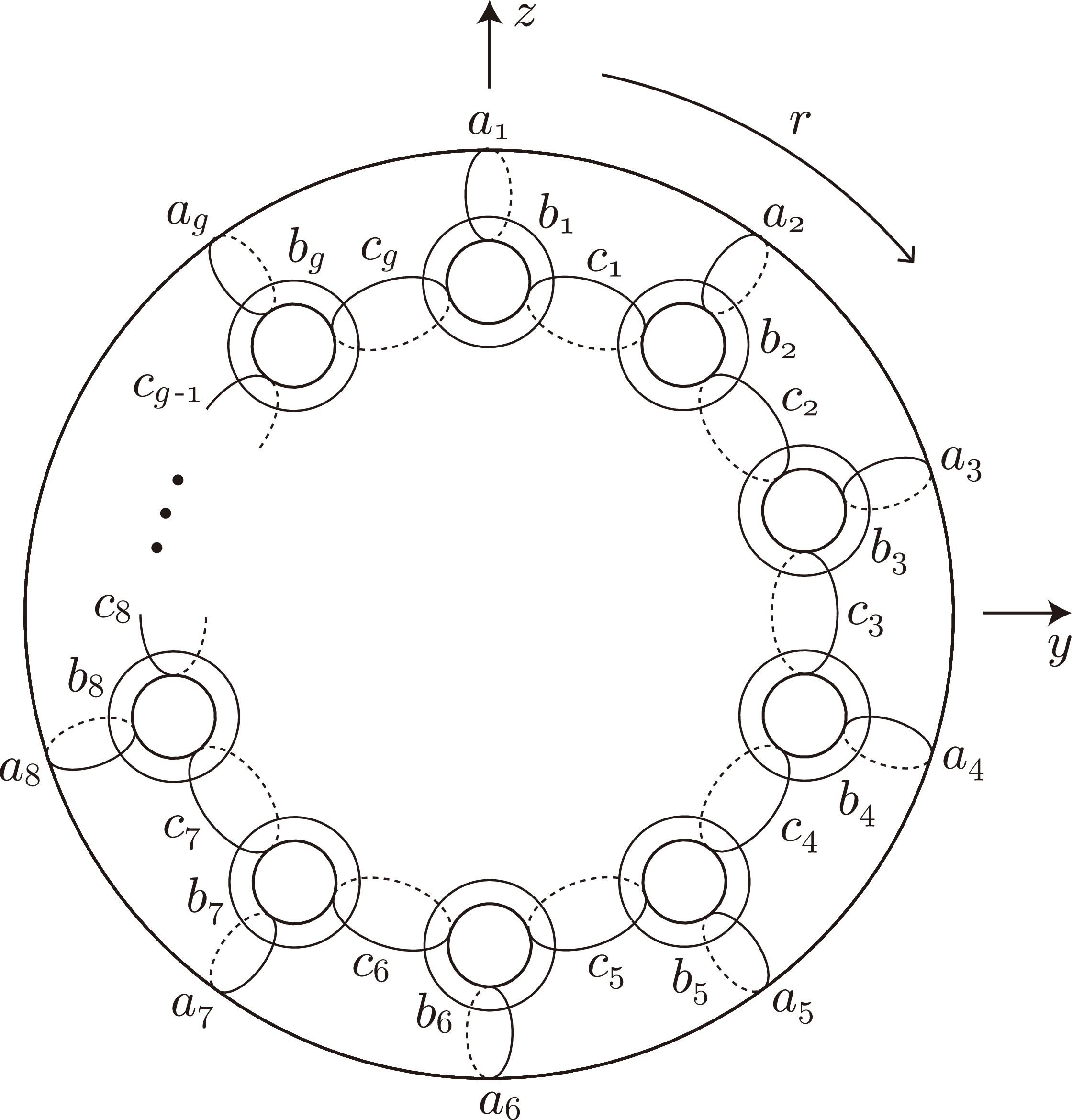}
       \caption{The rotation $r : \Sigma_g \to \Sigma_g$ by $-\frac{2\pi}{g}$ about the $x$-axis and the simple closed curves $a_i,b_i,c_i$ on $\Sigma_g$ for $i=1,2,\ldots,g$.}
       \label{rotation}
  \end{figure}
\begin{figure}[hbt]
  \centering
       \includegraphics[scale=.20]{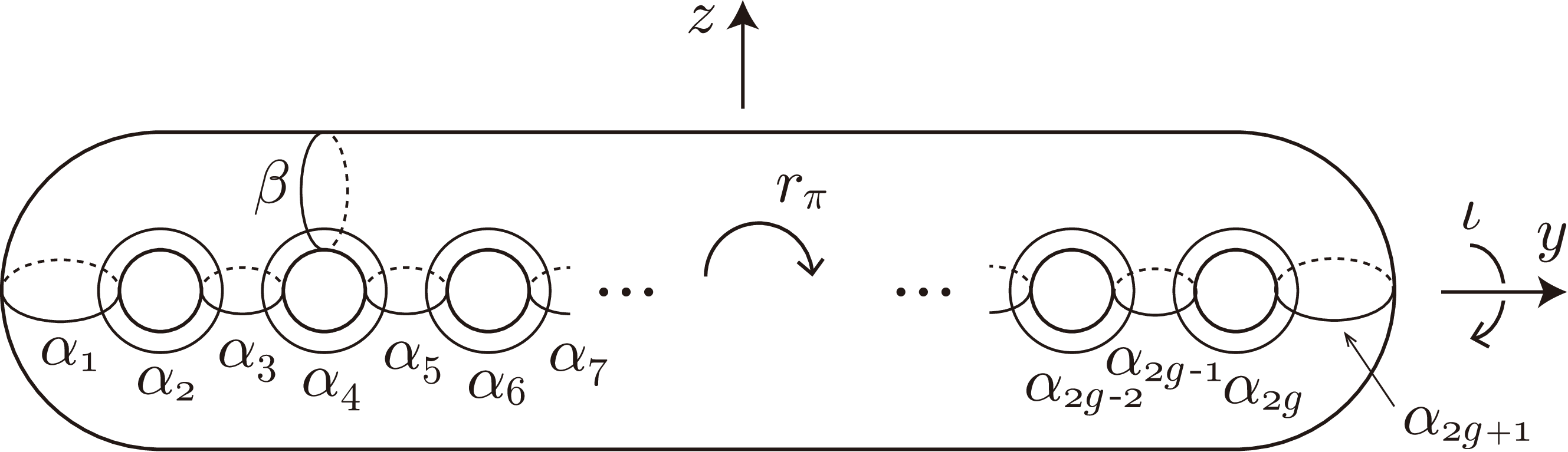}
       \caption{The hyperelliptic involution $\iota: \Sigma_g \to \Sigma_g$, the rotation $r_\pi :\Sigma_g \to \Sigma_g$ by 
       $\pi$ about the $x$-axis, and the simple closed curves $\alpha_1,\alpha_2,\ldots,\alpha_{2g+1},\beta$ on $\Sigma_g$.}
       \label{humphries1}
  \end{figure}

\begin{thm}[\cite{Li}]\label{Lickorishthm}
For $g\geq 1$, $\M$ is generated by $t_{a_1},\ldots,t_{a_g}, t_{b_1},\ldots,t_{b_g}$ and $t_{c_1},\ldots,t_{c_{g-1}}$. 
\end{thm}
\begin{thm}[\cite{Hu}]
For $g\geq 1$, $\M$ is generated by $t_{\alpha_1},t_{\alpha_2},\ldots,t_{\alpha_{2g}}$ and $t_\beta$. 
\end{thm}

Theorem~\ref{Penner} below is very useful to construct a pseudo-Anosov element. 
A \textit{multicurve} in $\Sigma_g$ is the union of a finite collection of disjoint essential simple closed curves in $\Sigma_g$. 
\begin{thm}[Theorem 3.1 in \cite{Pe}]\label{Penner}
Let $\mathcal{C} = \{\gamma_1,\gamma_2,\ldots,\gamma_m\}$ and $\mathcal{D} = \{\delta_1,\delta_2,\ldots,\delta_n\}$ be multicurves in $\Sigma_g$ such that no two component curves form a bigon and such that $\Sigma_g - \mathcal{C} \cup \mathcal{D}$ is a union of disks. 
Any product of positive powers of the Dehn twist $t_{\gamma_i}$ and negative powers of the Dehn twist $t_{\delta_j}$, where each $\gamma_i$ and each $\delta_j$ appear at least once, is pseudo-Anosov. 
\end{thm}

\begin{lem}\label{ex:1}
We define 
\begin{align*}
&h_{2i-1} = t_{\alpha_{2i-1}} t_{\alpha_{2i}}^{-1} \cdot t_{\alpha_{2i+1}} t_{\alpha_{2i+2}}^{-1} \cdots t_{\alpha_{2g-1}} t_{\alpha_{2g}}^{-1} \cdot t_{\alpha_{2g+1}}, \ \ \mathrm{and} \\ 
& 
h_{2i} = t_{\alpha_{2i}}^{-1} \cdot t_{\alpha_{2i+1}} t_{\alpha_{2i+2}}^{-1}  \cdots t_{\alpha_{2g-1}} t_{\alpha_{2g}}^{-1} \cdot t_{\alpha_{2g+1}}
\end{align*}
for $i=1,2,\ldots, g$ and $h_{2g+1} = t_{\alpha_{2g+1}}$. 
Moreover, we set 
\begin{align*}
&f_{2i-1} = t_\beta \cdot t_{\alpha_{2i-1}} t_{\alpha_{2i}}^{-1} \cdot t_{\alpha_{2i+1}} t_{\alpha_{2i+2}}^{-1} \cdots t_{\alpha_{2g-1}} t_{\alpha_{2g}}^{-1}, \ \ \mathrm{and} \\
&
f_{2i} = t_\beta \cdot t_{\alpha_{2i}}^{-1} \cdot t_{\alpha_{2i+1}} t_{\alpha_{2i+2}}^{-1} \cdots t_{\alpha_{2g-1}} t_{\alpha_{2g}}^{-1}
\end{align*}
for $i=1,2,\ldots,g$. 
Then, $h_1$ and $h_2$ are pseudo-Anosov elements in $\H$, and $f_1$ and $f_2$ are pseudo-Anosov elements in $\M$. 
\end{lem}
\begin{proof}
Since $t_{\alpha_1},t_{\alpha_2},\ldots,t_{\alpha_{2g+1}}$ are generators of $\H$ (see \cite{BH}), $h_1$ is also in $\H$. 
In the notation of Theorem 3.1 in \cite{Pe}, let $\mathscr{C} = \{\alpha_1,\alpha_3,\ldots,\alpha_{2g-1},\alpha_{2g+1}\}$ and $\mathscr{D} =\{\alpha_2,\alpha_4,\ldots,\alpha_{2g}\}$ be two multicurves on $\Sg$. 
Then, these two multicurves $\mathscr{C}$ and $\mathscr{D}$ satisfy the condition of Theorem~\ref{Penner}, and therefore, we see that $h_1$ is a pseudo-Anosov element. 
By applying these arguments again, with $\mathscr{C} = \{\alpha_1,\alpha_3,\ldots,\alpha_{2g-1},$ $\alpha_{2g+1}\}$ replaced by $\mathscr{C}' = \{\alpha_3,\alpha_5,\ldots,\alpha_{2g+1}\}$, we see that $h_2$ is also a pseudo-Anosov in $\H$. 

Similarly, when we take two multi-curves $\mathscr{C} = \{\alpha_1,\alpha_3,\ldots,\alpha_{2g-1},\beta\} $ (resp. $\mathscr{C} = \{\alpha_3,\alpha_5,\ldots,\alpha_{2g-1},\beta\}$) and $\mathscr{D} =\{\alpha_2,\alpha_4,\ldots,\alpha_{2g}\}$ for $f_1$ (resp. $f_2$), we see that $f_1$ (resp. $f_2$) is a pseudo-Anosov in $\M$. 
\end{proof}

Throughout this section, we use the notation $\prod_{i=1}^k f_i = f_1 f_2 \cdots f_k$ for a product $f_1f_2 \cdots f_k$ of $k$ elements $f_1,f_2,\ldots,f_k$ in $\M$. 
\begin{lem}\label{ex:2}
Set $\rho_n = r t_{c_1} t_{b_1}^{-n} t_{a_1}$ for a positive integer $n$.  
Then, $\rho_n$ is pseudo-Anosov in $\M$ for $g\geq 3$. 
Furthermore, for any positive real number $\lambda$, there is an integer $n$ 
such that the dilatation of $\rho_n$ is more than $\lambda$. 
\end{lem}
\begin{proof}
We note that $\rho_n^g = \left(\prod_{i=1}^g r^i t_{c_1} t_{b_1}^{-n} t_{a_1} r^{-i}\right) \cdot r^g$ and $r^g = \mathrm{id}$. This gives $\rho_n^g = \prod_{i=1}^g r^i t_{c_1} t_{b_1}^{-n} t_{a_1} r^{-i}$, and hence $\rho_n^g = \prod_{i=1}^g r^i t_{c_1} r^{-i} \cdot r^i t_{b_1}^{-n} r^{-i} \cdot r^i t_{a_1} r^{-i}$. 
Since $r(a_i, b_i, c_i) = (a_{i+1}, b_{i+1}, c_{i+1})$ for $i=1,2,\ldots,g-1$ and $r(a_g, b_g, c_g) = (a_1, b_1, c_1)$, we have
\begin{align*}
\rho_n^g = \prod_{i=1}^g t_{r^i(c_1)} t_{r^i(b_1)}^{-n} t_{r^i(a_1)} = t_{a_2} t_{b_2}^{-n} t_{c_2} \cdot t_{a_3} t_{b_3}^{-n} t_{c_3} \cdots t_{a_g} t_{b_g}^{-n} t_{c_g} \cdot t_{a_1} t_{b_1}^{-n} t_{c_1}. 
\end{align*}
If we take two multicurves $\mathscr{C} = \{a_1, c_1, a_2, c_2,\ldots, a_g, c_g\}$ and $\mathscr{D} = \{b_1, b_2,\ldots,b_g\}$ then 
by Theorem~\ref{Penner}, we see that $\rho_n^g$ is a pseudo-Anosov element. 
Therefore, $\rho_n$ is a pseudo-Anosov.

We will show that, for any positive real number $\lambda$, there is an integer $n$ such that the dilatation of $\rho_n$ is more than $\lambda$. 
Let $\tau$ be the train track on $\Sigma_g$ which carries each element of $\mathscr{C} \cup \mathscr{D}$. 
As shown in Figure~\ref{fig:circle-track}, the train track $\tau$ has branches $\delta_i$ for $i=1, \ldots, 3g$, $\mu_{3j-2}^k$ for $j=1,\ldots,g, k=1,2,3$, $\mu_{3j-1}$ for $j=1, \ldots, g$ and $\mu_{3j}^k$ for $j=1,\ldots,g, k=1,2$. 
For each branch, the transverse measure is indicated by the same symbol. 
The train track $\rho_n^g (\tau)$ is carried by $\tau$, and $t_{a_i} t_{b_i}^{-n} t_{c_i}(\tau)$ are also carried by $\tau$ for $i=1, \ldots, g$.  
The set $H$ of transverse mesures on $\tau$ determined by the equations $\mu_{3j-2}^1 = \mu_{3j-2}^2 = \mu_{3j-2}^3, \mu_{3j}^1 = \mu_{3j}^2$ is preserved by the action of  $t_{a_i}$ $t_{b_i}^{-1}$ and $t_{c_i}$. 
We denote $\mu_{3j-2}^1, \mu_{3j-2}^2, \mu_{3j-2}^3$ by $\mu_{3j-2}$ and $\mu_{3j}^1, \mu_{3j}^2$ by $\mu_{3j}$ for short. 
The actions of $\rho_n^g$ and $t_{a_i} t_{b_i}^{-n} t_{c_i}$ for $i=1, \ldots, g$ on $H$ are determined by the actions of them on values $\mu_j$ for $j=1,\ldots,3g$. 
For the homeomorphism $\varphi$ on $\Sigma_g$ such that $\varphi(\tau)$ is carried by $\tau$, let $M_{\varphi}$ be the incident matrix for $\varphi$, that is, the $(i,j)$ entry $(M_{\varphi})_{i,j}$ of $M_{\varphi}$ is the times $\varphi(\mu_j)$ passing through $\mu_i$. 
Then the matrix $M_{\rho_n^g}$ is a non-negative irreducible matrix and the maximal eigenvalue $\mathrm{PF}(M_{\rho_n^g})$ of $M_{\rho_n^g}$ equals the dilatation $\lambda_n^g$ of $\rho_n^g$, where $\lambda_n$ is the dilatation of $\rho_n$. 
Let $M$ and $N$ be matrices with the same size.  
If, for each $(i,j)$ entries $M_{i,j}$ and $N_{i,j}$ of $M$ and $N$ respectively, $M_{i.j} \geq N_{i,j}$, then we write $M \geq N$. 
If $\psi = t_{a_i}$ or $t_{b_i}^{-1}$ or $t_{c_i}$, then $M_{\psi} = I_{3g} + P$, where $I_{3g}$ is the identity matrix of order $3g$ and $P$ is a matrix all of whose entries are non-negative. 
Therefore, 
\begin{align*}
&M_{\rho_n^g} \geq M_{t_{b_1}^{-n}} M_{t_{c_1}}  \\ 
&= \left(
\begin{array}{c|c|c|c}
{\begin{array}{ccc} 1 & n & n \\ 0 & 1 & 0 \\ 0 & 0 & 1 \end{array}} 
&  {\begin{array}{ccc} 0 & 0 & 0 \\ 0 & 0 & 0 \\ 0 & 0 & 0 \end{array}}
& O 
&  {\begin{array}{ccc} 0 & 0 & n \\ 0 & 0 & 0 \\ 0 & 0 & 0 \end{array}} \\ \hline 
{\begin{array}{ccc} 0 & 0 & 0 \\ 0 & 0 & 0 \\ 0 & 0 & 0 \end{array}} 
&  {\begin{array}{ccc} 1 & 0 & 0 \\ 0 & 1 & 0 \\ 0 & 0 & 1 \end{array}}
& O 
&  {\begin{array}{ccc} 0 & 0 & 0 \\ 0 & 0 & 0 \\ 0 & 0 & 0 \end{array}} \\ \hline
O & O & {\begin{array}{c} \\ I_{3g-9} \\  \\ \end{array}} & O \\ \hline
{\begin{array}{ccc} 0 & 0 & 0 \\ 0 & 0 & 0 \\ 0 & 0 & 0 \end{array}} 
&  {\begin{array}{ccc} 0 & 0 & 0 \\ 0 & 0 & 0 \\ 0 & 0 & 0 \end{array}}
& O 
&  {\begin{array}{ccc} 1 & 0 & 0 \\ 0 & 1 & 0 \\ 0 & 0 & 1 \end{array}}
\end{array}
\right)
\left(
\begin{array}{c|c|c|c}
{\begin{array}{ccc} 1 & 0 & 0 \\ 0 & 1 & 0 \\ 1 & 0 & 1 \end{array}} 
&  {\begin{array}{ccc} 0 & 0 & 0 \\ 0 & 0 & 0 \\ 1 & 0 & 0 \end{array}}
& O 
&  {\begin{array}{ccc} 0 & 0 & 0 \\ 0 & 0 & 0 \\ 0 & 0 & 0 \end{array}} \\ \hline 
{\begin{array}{ccc} 0 & 0 & 0 \\ 0 & 0 & 0 \\ 0 & 0 & 0 \end{array}} 
&  {\begin{array}{ccc} 1 & 0 & 0 \\ 0 & 1 & 0 \\ 0 & 0 & 1 \end{array}}
& O 
&  {\begin{array}{ccc} 0 & 0 & 0 \\ 0 & 0 & 0 \\ 0 & 0 & 0 \end{array}} \\ \hline
O & O & {\begin{array}{c} \\ I_{3g-9} \\  \\ \end{array}} & O \\ \hline
{\begin{array}{ccc} 0 & 0 & 0 \\ 0 & 0 & 0 \\ 0 & 0 & 0 \end{array}} 
&  {\begin{array}{ccc} 0 & 0 & 0 \\ 0 & 0 & 0 \\ 0 & 0 & 0 \end{array}}
& O 
&  {\begin{array}{ccc} 1 & 0 & 0 \\ 0 & 1 & 0 \\ 0 & 0 & 1 \end{array}}
\end{array}
\right) \\
&= 
\left(
\begin{array}{c|c|c|c}
{\begin{array}{ccc} n+1 & n & n \\ 0 & 1 & 0 \\ 1 & 0 & 1 \end{array}} 
&  {\begin{array}{ccc} n & 0 & 0 \\ 0 & 0 & 0 \\ 1 & 0 & 0 \end{array}}
& O 
&  {\begin{array}{ccc} 0 & 0 & n \\ 0 & 0 & 0 \\ 0 & 0 & 0 \end{array}} \\ \hline 
{\begin{array}{ccc} 0 & 0 & 0 \\ 0 & 0 & 0 \\ 0 & 0 & 0 \end{array}} 
&  {\begin{array}{ccc} 1 & 0 & 0 \\ 0 & 1 & 0 \\ 0 & 0 & 1 \end{array}}
& O 
&  {\begin{array}{ccc} 0 & 0 & 0 \\ 0 & 0 & 0 \\ 0 & 0 & 0 \end{array}} \\ \hline
O & O & {\begin{array}{c} \\ I_{3g-9} \\  \\ \end{array}} & O \\ \hline
{\begin{array}{ccc} 0 & 0 & 0 \\ 0 & 0 & 0 \\ 0 & 0 & 0 \end{array}} 
&  {\begin{array}{ccc} 0 & 0 & 0 \\ 0 & 0 & 0 \\ 0 & 0 & 0 \end{array}}
& O 
&  {\begin{array}{ccc} 1 & 0 & 0 \\ 0 & 1 & 0 \\ 0 & 0 & 1 \end{array}}
\end{array}
\right)
\geq 
\left(
\begin{array}{c|c|c|c}
{\begin{array}{ccc} n+1 & 0 & 0 \\ 0 & 1 & 0 \\ 0 & 0 & 1 \end{array}} 
&  {\begin{array}{ccc} 0 & 0 & 0 \\ 0 & 0 & 0 \\ 0 & 0 & 0 \end{array}}
& O 
&  {\begin{array}{ccc} 0 & 0 & 0 \\ 0 & 0 & 0 \\ 0 & 0 & 0 \end{array}} \\ \hline 
{\begin{array}{ccc} 0 & 0 & 0 \\ 0 & 0 & 0 \\ 0 & 0 & 0 \end{array}} 
&  {\begin{array}{ccc} 1 & 0 & 0 \\ 0 & 1 & 0 \\ 0 & 0 & 1 \end{array}}
& O 
&  {\begin{array}{ccc} 0 & 0 & 0 \\ 0 & 0 & 0 \\ 0 & 0 & 0 \end{array}} \\ \hline
O & O & {\begin{array}{c} \\ I_{3g-9} \\  \\ \end{array}} & O \\ \hline
{\begin{array}{ccc} 0 & 0 & 0 \\ 0 & 0 & 0 \\ 0 & 0 & 0 \end{array}} 
&  {\begin{array}{ccc} 0 & 0 & 0 \\ 0 & 0 & 0 \\ 0 & 0 & 0 \end{array}}
& O 
&  {\begin{array}{ccc} 1 & 0 & 0 \\ 0 & 1 & 0 \\ 0 & 0 & 1 \end{array}}
\end{array}
\right).
\end{align*}
It is well-known that, if $M$ is a non-negative irreducible matrix and $N$ is a non-negative matrix 
such that $M \geq N$, then every eigenvalue $\sigma$ of $N$ satisfies 
$ \mathrm{PF}(M) \geq | \sigma |$
(see Theorem 9.1.1 of \cite{Sternberg}, for example). 
Therefore, $\lambda_n^g \geq n+1$.  
\end{proof}
\begin{figure}[hbt]
  \centering
\includegraphics[height=5cm]{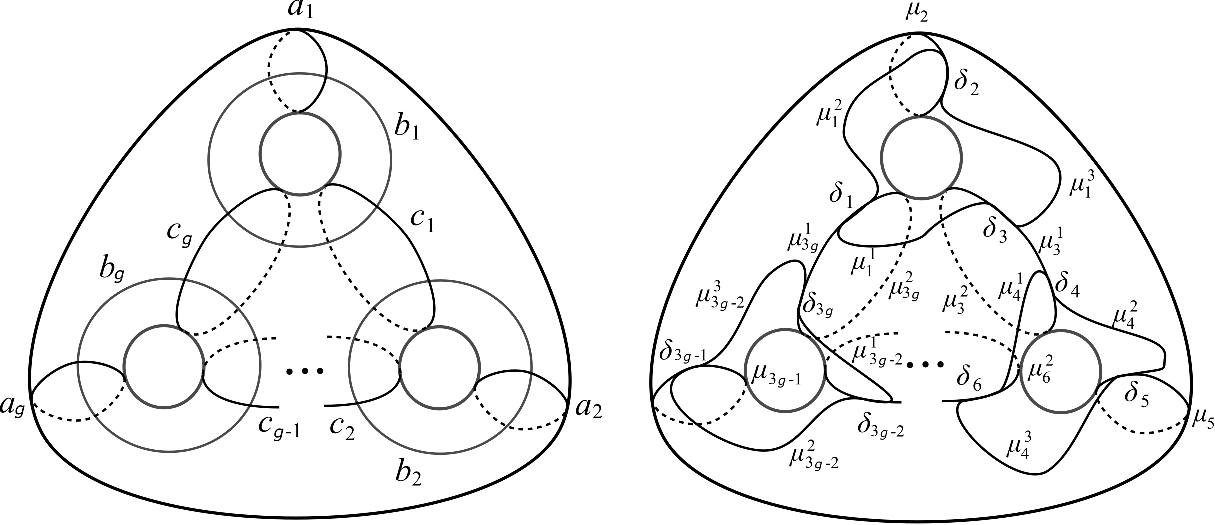}
\caption{The $x$-axis crosses this paper from the back side to this side.   LEFT: simple closed curves $a_i, b_i, c_i$ $(i=1,2, \ldots, g)$ 
RIGHT: a train track on $\Sigma_g$ carries simple closed curves 
$a_i, b_i, c_i$ $(i=1,2, \ldots, g)$} 
\label{fig:circle-track}
\end{figure}

\begin{lem}\label{ex:3}
Set $\rho = rt_{a_1}$ and $\rho' = rt_{c_1}$. Then, $\rho$ and $\rho'$ are reducible but not periodic in $\M$. 
\end{lem}
\begin{proof}
We have 
\begin{align*}
(\rho)^g &= (rt_{a_1})^g 
= \left( \prod_{i=1}^g r^i t_{a_1} r^{-i} \right) \cdot r^g 
= \left( \prod_{i=1}^g t_{r^i(a_1)} \right)
= t_{a_2} t_{a_3} \cdots t_{a_g} t_{a_1},  
\end{align*}
and hence we see that for any integer $n$, $(\rho)^{gn} = t_{a_1}^n \cdots t_{a_g}^n$. 
This means that $\rho$ is a reducible and not a periodic element. 
Similar arguments apply to $\rho'$. 
\end{proof}

\begin{lem}\label{ex:4}
Set $\mathcal{R} = (t_{\alpha_1}^2 t_{\alpha_2} t_{\alpha_3} \cdots t_{\alpha_{2g-1}})^{g-1} t_{\alpha_{2g+1}}^{-1}$. 
Then, $\mathcal{R}$ is reducible but not periodic in $\H$. 
In particular, the relation $t_{\alpha_{2g+1}} = \mathcal{R}^{-(2g-1)}$ holds. 
\end{lem}
The element $\mathcal{R}^{-1}$ is called a \textit{root} of Dehn twist of degree $2g-1$ (see \cite{MS,MR,Monden}). 
\begin{proof}
Note that $\mathcal{R}$ is in $\H$ since $t_{\alpha_1},t_{\alpha_2},\ldots,t_{\alpha_{2g+1}}$ are generators of $\H$. 
It is well-known 
that the relation $(t_{\alpha_1}^2 t_{\alpha_2} t_{\alpha_3} \cdots t_{\alpha_{2g-1}})^{2g-1} = t_{\alpha_{2g+1}}^2$ holds in $\H$ (see Section 4.4.1 in \cite{FM}). 
This relation is called the \textit{chain relation}. 
Since $\alpha_{2g+1}$ is disjoint from $\alpha_1,\ldots,\alpha_{2g-1}$, 
by the chain relation, we have 
\[\mathcal{R}^{2g-1} =  (t_{\alpha_1}^2 t_{\alpha_2} t_{\alpha_3} \cdots t_{\alpha_{2g-1}})^{(g-1)(2g-1)} t_{\alpha_{2g+1}}^{-(2g-1)} = t_{\alpha_{2g+1}}^{2(g-1)} t_{\alpha_{2g+1}}^{-(2g-1)} = t_{\alpha_{2g+1}}^{-1}.\] 
Since a Dehn twist in $\H$ is reducible but not periodic in $\H$, $\mathcal{R}$ is also reducible but not periodic. 
\end{proof}

\section{Proofs of Theorems~\ref{thm:1} and~\ref{thm:50}}
In this section, we prove Theorems~\ref{thm:1} and ~\ref{thm:50} in the introduction. 
We will repeatedly use the fact that if a Dehn twist $t_c$ about a simple closed curve $c$ and an element $f$ are in a certain subgroup $G$ of $\M$, then $t_f(c)$ is also in $G$ by the relation $t_{f(c)} = f t_c f^{-1}$.
Recall that the notations $\alpha_1,\alpha_2,\ldots,\alpha_{2g+1}$ and $\beta$ denote the simple closed curves on $\Sigma_g$ as in Figure~\ref{humphries1}. 
Note that $\alpha_i$ (resp. $\beta$) intersects $\alpha_{i+1}$ (resp. $\alpha_4$) transversely at exactly one point for $i=1,2,\ldots,2g-1$. 
Moreover, $\alpha_i$ (resp. $\beta$) is disjoint from $\alpha_j$ if $|i-j|>1$ (resp. $j\neq 4$). 
In this section, we will use the following 
fact:
\begin{align*}
&t_{\alpha_{i+1}}^{-1}(\alpha_i) = t_{\alpha_i}(\alpha_{i+1}),& &\mathrm{and}& &t_{\alpha_{i+1}}(\alpha_i) = t_{\alpha_i}^{-1}(\alpha_{i+1}).& 
\end{align*}
We denote by $\gamma_{i}$ the simple closed curve $t_{\alpha_{i}}(\alpha_{i+1})$ for $i=4,5,\ldots,2g-1$.

Theorem~\ref{thm:1} follows from Theorems~\ref{thm:5} and ~\ref{thm:6} below. 
\begin{thm}\label{thm:5}
Let $H$ be the subgroup of $\H$ generated by $h_1$ and $h_2$ defined in Lemma~\ref{ex:1}. 
Then, $H=\H$ for $g\geq 1$.  
\end{thm}
\begin{proof}
We see that $t_{\alpha_1} = h_1 h_2^{-1}$ is in $H$, 
and hence 
$t_{h_2(\alpha_1)}$ is in $H$. 
Since 
$h_2(\alpha_1) = t_{\alpha_2}^{-1}(\alpha_1) = t_{\alpha_1}(\alpha_2)$, 
we obtain 
$t_{h_2(\alpha_1)} = t_{t_{\alpha_1}(\alpha_2)}$ 
and therefore, $t_{\alpha_2}$ is in $H$. 
We set $h_3 = t_{\alpha_2} h_2$, then $h_3$ is in $H$, and hence $t_{h_3(\alpha_2)}$ is in $H$. 
By $h_3(\alpha_2) = t_{\alpha_3}(\alpha_2) = t_{\alpha_2}^{-1}(\alpha_3)$, 
we have 
$t_{h_3(\alpha_2)} = t_{t_{\alpha_2}^{-1}(\alpha_3)}$, 
and therefore $t_{\alpha_3}$ is in $H$. 
Repeating this argument, we can show that $t_{\alpha_1},\ldots,t_{\alpha_{2g+1}}$ are contained in $H$. 
Since $\H$ is generated by these Dehn twists, we see that $H=\H$, and the proof is complete. 
\end{proof}

\begin{thm}\label{thm:6}
Let $G$ be the subgroup of $\M$ generated by $f_1$ and $f_2$ defined in Lemma~\ref{ex:1}. 
Then, $G = \M$ for $g\geq 3$. 
\end{thm}

Lemmas~\ref{lem:5} and~\ref{lem:6} below are used to show Theorem~\ref{thm:6}. 
\begin{lem}\label{lem:5}
For $g\geq 3$, the elements $t_{\alpha_1}, t_{\alpha_2}, t_{\alpha_3},  
t_{\gamma_{4}}, t_{\gamma_{5}}, 
f_3, f_4$ and $f_5$ are contained in $G$. 
\end{lem}
\begin{proof}
By $t_{\alpha_1} = t_\beta t_{\alpha_1} t_\beta^{-1} = f_1f_2^{-1}$, $t_{\alpha_1}$ is in $G$, and hence $t_{\alpha_1}^{-1}f_2$ is also in $G$. 
Since we find that $f_2(\alpha_1) = t_{\alpha_2}^{-1}(\alpha_1) = t_{\alpha_1}(\alpha_2)$, we have $t_{\alpha_1}^{-1} f_2 (\alpha_1) = \alpha_2$. 
Therefore, 
$t_{\alpha_2}$ is in $G$. 
By $f_3 = t_{\alpha_2} f_2$, we see that $f_3$ is in $G$. 
Since we find that $f_3(\alpha_2) = t_{\alpha_3}(\alpha_2) = t_{\alpha_2}^{-1}(\alpha_3)$, we have $t_{\alpha_2} f_3 (\alpha_2) = \alpha_3$, and therefore we see that $t_{\alpha_3}$ is in $G$. 
From this, $f_4 = t_{\alpha_3}^{-1} f_3$ is in $G$. 
Note that we find that $f_4(\alpha_3) = t_\beta t_{\alpha_4}^{-1}(\alpha_3) = t_\beta t_{\alpha_3}(\alpha_4) = t_{\alpha_3} t_\beta (\alpha_4)$. 
This gives $t_{\alpha_3}^{-1} f_4 (\alpha_3) = t_\beta(\alpha_4)$, and hence 
$t_{\delta_4}$ is in $G$, 
where $\delta_4$ is the simple closed curve $t_\beta(\alpha_4)$.  
Furthermore, the following holds:
\begin{align*} 
f_4 &= t_\beta t_{\alpha_4}^{-1} \cdot t_{\alpha_5} t_{\alpha_6}^{-1} t_{\alpha_7} t_{\alpha_8}^{-1} \cdots t_{\alpha_{2g-1}} t_{\alpha_{2g}}^{-1} \\
&= t_\beta t_{\alpha_4}^{-1} t_{\beta}^{-1} \cdot t_\beta t_{\alpha_5} t_{\alpha_6}^{-1} t_{\alpha_7} t_{\alpha_8}^{-1} \cdots t_{\alpha_{2g-1}} t_{\alpha_{2g}}^{-1} \\
&= t_{\delta_4}^{-1} f_5.   
\end{align*}
This means that $f_5$ is in $G$. 
Since $\beta$ and $\alpha_4$ are disjont from $\alpha_6,\ldots,\alpha_{2g}$, we obtain 
$f_5(
\delta_4 
) = t_\beta t_{\alpha_5}(
\delta_4
)$.  
This gives
\begin{align*}
t_{\delta_4} f_5(\delta_4) 
& = t_\beta t_{\alpha_4} t_\beta^{-1} t_\beta t_{\alpha_5} t_\beta(\alpha_4) \\
& = t_\beta t_{\alpha_4} t_\beta t_{\alpha_5} (\alpha_4) \\
&= t_{\alpha_4} t_\beta t_{\alpha_4} t_{\alpha_4}^{-1} (\alpha_5) \\
&= t_{\alpha_4}(\alpha_5) (= \gamma_{4}). 
\end{align*}
Therefore, $t_{\gamma_{4}}$ is in $G$.  
Moreover, we have 
\begin{align*}
f_3(\gamma_{4}) &= t_\beta t_{\alpha_3} t_{\alpha_4}^{-1} t_{\alpha_5} t_{\alpha_6}^{-1} t_{\alpha_4}(\alpha_5) \\
&= t_\beta t_{\alpha_3} t_{\alpha_4}^{-1} t_{\alpha_5} t_{\alpha_4} t_{\alpha_6}^{-1}(\alpha_5) \\
&= t_\beta t_{\alpha_3} t_{\alpha_4}^{-1} t_{\alpha_5} t_{\alpha_4} t_{\alpha_5}(\alpha_6) \\
&= t_\beta t_{\alpha_3} t_{\alpha_4}^{-1} t_{\alpha_4} t_{\alpha_5} t_{\alpha_4}(\alpha_6) \\
&= t_\beta t_{\alpha_3} t_{\alpha_5}(\alpha_6) \\
&= t_{\alpha_5}(\alpha_6) (= \gamma_{5}). 
\end{align*}
Therefore, $t_{\gamma_{5}}$ is in $G$, and the proof is complete.
\end{proof}

\begin{lem}\label{lem:6}
Let $g\geq 4$, and set $F_{2k-1} = t_{\alpha_3} t_{\alpha_5} \cdots t_{\alpha_{2k-3}} f_{2k-1}$ for $3\leq k \leq g$. 
We assume that the elements $
t_{\gamma_{2i-1}} 
, F_{2i-1}, F_{2i+1}$ are contained in $G$ for some $i$ such that $3\leq i \leq g-1$, then the following holds:  
\begin{enumerate}
\item[(1)] 
$t_{\gamma_{2i}}$ and $t_{\gamma_{2i+1}}$ are in $G$ 
if $3\leq i \leq g-1$, 
\item[(2)] 
$F_{2i+3}$ is in $G$ 
if $i \leq g-2$, and 
\item[(3)] 
$t_{\alpha_{2g}}$ is in $G$ 
if $i=g-1$. 
\end{enumerate}
\end{lem}
\begin{proof}
First, we show the former part of Lemma~\ref{lem:6} (1) (i.e., the element $t_{\gamma_{2i}}$ 
is in $G$ for $3\leq i \leq g-1$). 
Note that the following holds:
\begin{align*}
f_{2i+1} (\gamma_{2i-1}) &=  t_\beta t_{\alpha_{2i+1}} t_{\alpha_{2i+2}}^{-1} \cdots t_{\alpha_{2g-1}} t_{\alpha_{2g}}^{-1}  (t_{\alpha_{2i-1}}(\alpha_{2i})) = t_\beta t_{\alpha_{2i+1}} (t_{\alpha_{2i-1}}(\alpha_{2i})). 
\end{align*}
Therefore, by $3\leq i$, we have $f_{2i+1} (\gamma_{2i-1}) = t_{\alpha_{2i+1}} t_{\alpha_{2i-1}}(\alpha_{2i})$. 
This gives 
\begin{align*}
F_{2i+1} (\gamma_{2i-1})
&=  t_{\alpha_3} t_{\alpha_5} \cdots t_{\alpha_{2i-3}} t_{\alpha_{2i-1}} (t_{\alpha_{2i+1}} t_{\alpha_{2i-1}}(\alpha_{2i})) \\
&= t_{\alpha_{2i-1}} (t_{\alpha_{2i+1}} t_{\alpha_{2i-1}}(\alpha_{2i})). 
\end{align*}
From the above arguments with $t_{\alpha_{2i}} t_{\alpha_{2i+1}} (\alpha_{2i}) = \alpha_{2i+1}$, we have 
\begin{align*}
t_{\gamma_{2i-1}} F_{2i+1} (\gamma_{2i-1}) & = t_{\alpha_{2i-1}} t_{\alpha_{2i}} t_{\alpha_{2i-1}}^{-1} t_{\alpha_{2i-1}} t_{\alpha_{2i+1}} t_{\alpha_{2i-1}} (\alpha_{2i}) \\
&= t_{\alpha_{2i-1}} t_{\alpha_{2i}}  t_{\alpha_{2i-1}} t_{\alpha_{2i+1}} (\alpha_{2i}) \\
&= t_{\alpha_{2i}} t_{\alpha_{2i-1}}  t_{\alpha_{2i}} t_{\alpha_{2i+1}} (\alpha_{2i}) \\
&= t_{\alpha_{2i}} t_{\alpha_{2i-1}} (\alpha_{2i+1}) \\
&= t_{\alpha_{2i}} (\alpha_{2i+1}) (=\gamma_{2i}), 
\end{align*}
and hence the element $t_{\gamma_{2i}}$ is in $G$. 

Second, we show the latter part of Lemma~\ref{lem:6} (1) (i.e., the element $t_{\gamma_{2i+1}}$ is in $G$ for $3\leq i\leq g-1$). 
By $3 \leq i$, we have
\begin{align*}
f_{2i-1} (\gamma_{2i}) &=  t_\beta t_{\alpha_{2i-1}} t_{\alpha_{2i}}^{-1} \cdots t_{\alpha_{2g-1}} t_{\alpha_{2g}}^{-1} (t_{\alpha_{2i}}(\alpha_{2i+1})) \\
&= t_{\alpha_{2i-1}} t_{\alpha_{2i}}^{-1} \cdots t_{\alpha_{2g-1}} t_{\alpha_{2g}}^{-1} t_{\alpha_{2i}}(\alpha_{2i+1}) \\
&= t_{\alpha_{2i-1}} t_{\alpha_{2i}}^{-1} t_{\alpha_{2i+1}} t_{\alpha_{2i+2}}^{-1} t_{\alpha_{2i}} (\alpha_{2i+1}) \\
&= t_{\alpha_{2i-1}} t_{\alpha_{2i}}^{-1} t_{\alpha_{2i+1}} t_{\alpha_{2i}} t_{\alpha_{2i+2}}^{-1} (\alpha_{2i+1}) \\
&= t_{\alpha_{2i-1}} t_{\alpha_{2i}}^{-1} t_{\alpha_{2i+1}} t_{\alpha_{2i}} t_{\alpha_{2i+1}} (\alpha_{2i+2}). 
\end{align*}
Moreover, by $t_{\alpha_{2i-1}} t_{\alpha_{2i}}^{-1} t_{\alpha_{2i+1}} t_{\alpha_{2i}} t_{\alpha_{2i+1}} 
= t_{\alpha_{2i-1}} t_{\alpha_{2i}}^{-1} t_{\alpha_{2i}} t_{\alpha_{2i+1}} t_{\alpha_{2i}} 
= t_{\alpha_{2i+1}} t_{\alpha_{2i-1}} t_{\alpha_{2i}}$, we obtain
\begin{align*}
f_{2i-1} (\gamma_{2i}) = t_{\alpha_{2i+1}} t_{\alpha_{2i-1}} t_{\alpha_{2i}}(\alpha_{2i+2}) = t_{\alpha_{2i+1}}(\alpha_{2i+2}), 
\end{align*}
and moreover, 
we have
\begin{align*}
F_{2i-1} (\gamma_{2i}) &= t_{\alpha_3} t_{\alpha_5} \cdots t_{\alpha_{2i-3}} f_{2i-1} (\gamma_{2i}) \\
&= t_{\alpha_3} t_{\alpha_5} \cdots t_{\alpha_{2i-3}} (t_{\alpha_{2i+1}}(\alpha_{2i+2})) \\
&= t_{\alpha_{2i+1}}(\alpha_{2i+2}) (=\gamma_{2i+1}). 
\end{align*}
Therefore, the element $t_{\gamma_{2i+1}}$ is in $G$. 

Third, we show Lemma~\ref{lem:6} (2) (i.e., the element $F_{2i+3}$ is in $G$ for $i \leq g-2$). 
By $3 \leq i$, we obtain
\begin{align*}
f_{2i+1} &=  t_\beta t_{\alpha_{2i+1}} t_{\alpha_{2i+2}}^{-1} t_{\alpha_{2i+3}} t_{\alpha_{2i+4}}^{-1} \cdots t_{\alpha_{2g-1}} t_{\alpha_{2g}}^{-1} \\
&= t_{\alpha_{2i+1}} t_{\alpha_{2i+2}}^{-1} t_{\alpha_{2i+1}}^{-1} t_{\alpha_{2i+1}} \cdot t_\beta t_{\alpha_{2i+3}} t_{\alpha_{2i+4}}^{-1} \cdots t_{\alpha_{2g-1}} t_{\alpha_{2g}}^{-1} \\
&= t_{\alpha_{2i+1}} t_{\alpha_{2i+2}}^{-1} t_{\alpha_{2i+1}}^{-1} \cdot t_{\alpha_{2i+1}} \cdot f_{2i+3}. 
\end{align*}
Therefore, we have
\begin{align*}
F_{2i+1} &= t_{\alpha_3} t_{\alpha_5} \cdots t_{\alpha_{2i-3}} t_{\alpha_{2i-1}} \cdot t_{\alpha_{2i+1}} t_{\alpha_{2i+2}}^{-1} t_{\alpha_{2i+1}}^{-1} \cdot t_{\alpha_{2i+1}} \cdot f_{2i+3} \\
&= t_{\alpha_{2i+1}} t_{\alpha_{2i+2}}^{-1} t_{\alpha_{2i+1}}^{-1} \cdot t_{\alpha_3} t_{\alpha_5} \cdots t_{\alpha_{2i-3}} t_{\alpha_{2i-1}} \cdot t_{\alpha_{2i+1}} \cdot f_{2i+3} \\
&= t_{\gamma_{2i+1}}^{-1} F_{2i+3}. 
\end{align*}
Since $t_{\gamma_{2i+1}}$ is in $G$, the element $F_{2i+3}$ is in $G$. 

Finally, we show Lemma~\ref{lem:6} (3) (i.e., the element $t_{\alpha_{2g}}$ is in $G$ for $i=g-1$). 
If $i=g-1$, then $F_{2i+1} = F_{2g-1} = t_{\alpha_3} t_{\alpha_5} \cdots t_{\alpha_{2g-3}} f_{2g-1}$ is in $G$. 
Then, we see that the following holds: 
\begin{align*}
F_{2g-1} &= t_{\alpha_3} t_{\alpha_5} \cdots t_{\alpha_{2g-3}} t_\beta t_{\alpha_{2g-1}} t_{\alpha_{2g}}^{-1} \\
&= t_{\alpha_3} t_{\alpha_5} \cdots t_{\alpha_{2g-3}} t_\beta \cdot t_{\alpha_{2g-1}} t_{\alpha_{2g}}^{-1} t_{\alpha_{2g-1}}^{-1} t_{\alpha_{2g-1}} \\
&= t_{\alpha_{2g-1}} t_{\alpha_{2g}}^{-1} t_{\alpha_{2g-1}}^{-1} \cdot t_{\alpha_3} t_{\alpha_5} \cdots t_{\alpha_{2g-3}} t_\beta t_{\alpha_{2g-1}} \\
&= t_{\gamma_{2g-1}} 
\cdot t_{\alpha_3} t_{\alpha_5} \cdots t_{\alpha_{2g-3}} t_\beta t_{\alpha_{2g-1}}.
\end{align*}
This means that $t_{\alpha_3} t_{\alpha_5} \cdots t_{\alpha_{2g-3}} t_\beta t_{\alpha_{2g-1}}$ is in $G$ since $t_{\gamma_{2i+1}} = t_{\gamma_{2g-1}}$ is in $G$. 
Therefore, the element $t_{\alpha_{2g}}$ is in $G$ by $F_{2g-1} = (t_{\alpha_3} t_{\alpha_5} \cdots t_{\alpha_{2g-3}} t_\beta t_{\alpha_{2g-1}}) t_{\alpha_{2g}}^{-1}$, and the lemma follows. 
\end{proof}

\begin{proof}[Proof of Theorem~\ref{thm:6}] 
We note that we have seen that the elements $t_{\alpha_1}$, $t_{\alpha_2}$, $t_{\alpha_3}$, $t_{\gamma_4}$, $t_{\gamma_5}$ and $f_5$ are contained in $G$ by Lemma~\ref{lem:5}. 

Suppose that $g=3$. 
Then, we have
\begin{align*}
f_5 = t_\beta t_{\alpha_5} t_{\alpha_6}^{-1} = t_{\alpha_5} t_{\alpha_6}^{-1} t_{\alpha_5}^{-1} t_\beta t_{\alpha_5} = t_{\gamma_5}^{-1} \cdot t_\beta t_{\alpha_5}, 
\end{align*}
and hence the element $t_\beta t_{\alpha_5}$ is in $G$. 
Moreover, $t_{\alpha_6}$ is in $G$ by $f_5 = (t_\beta t_{\alpha_5}) t_{\alpha_6}^{-1}$, $t_{\alpha_5}$ is in $G$ by $\gamma_5 = t_{\alpha_5}(\alpha_6) = t_{\alpha_6}^{-1}(\alpha_5)$, and $t_{\alpha_4}$ is in $G$ by $\gamma_4 = t_{\alpha_4}(\alpha_5) = t_{\alpha_5}^{-1}(\alpha_4)$. 
Therefore, the elements $t_{\alpha_1},t_{\alpha_2},\ldots,t_{\alpha_6}$ are contained in $G$, and hence we see that $t_\beta$ is in  $G$ by $f_1 = t_\beta t_{\alpha_1} t_{\alpha_2}^{-1} \cdots t_{\alpha_5} t_{\alpha_6}^{-1}$. 
Since $\M$ is generated by $t_{\alpha_1},t_{\alpha_2},\ldots,t_{\alpha_6},t_\beta$ if $g=3$, we see that $G=\M$. 

Suppose that $g\geq 4$. 
Then, we have
\begin{align*}
f_5 &= t_\beta t_{\alpha_5} t_{\alpha_6}^{-1} t_{\alpha_7} t_{\alpha_8}^{-1} \cdots t_{\alpha_{2g-1}} t_{\alpha_{2g}}^{-1} \\
&= t_{\alpha_5} t_{\alpha_6}^{-1} t_{\alpha_5}^{-1} t_{\alpha_5} \cdot t_\beta  t_{\alpha_7} t_{\alpha_8}^{-1} \cdots t_{\alpha_{2g-1}} t_{\alpha_{2g}}^{-1} \\
&= t_{\gamma_5}^{-1} t_{\alpha_5} f_7. 
\end{align*}
This means that the element $t_{\alpha_5} f_7$ is contained in $G$. 
Moreover, since $t_{\alpha_3}$ and $f_5$ are in $G$, $F_5 = t_{\alpha_3} f_5$ and $F_7 = t_{\alpha_3} t_{\alpha_5} f_7$ are in $G$. 
Therefore, the condition of Lemma~\ref{lem:6} is satisfied for the case where $i=3$. 
By applying Lemma~\ref{lem:6} repeatedly, we see that $t_{\gamma_j}$ is in $G$ for $j=6,7,\ldots,2g-1$ and $t_{\alpha_{2g}}\in G$, and hence the element $t_{\alpha_{2g-1}}$ is in $G$ by $\gamma_{2g-1} = t_{\alpha_{2g-1}(\alpha_{2g})} = t_{\alpha_{2g}}^{-1}(\alpha_{2g-1})$. 
By repeating this argument, we find that the elements $t_{\alpha_4},t_{\alpha_5},\ldots,t_{\alpha_{2g}}$ are contained in $G$. 
Moreover, from the definition of $f_1$, we have $t_\beta \in G$, and hence $G=\M$.

This finishes the proof. 
\end{proof}

Theorem~\ref{thm:50} follows from Theorem~\ref{thm:51} below.

\begin{thm}\label{thm:51}
Let $H'$ be the subgroup of $\H$ generated by $\mathcal{R}$ 
and $r_\pi \mathcal{R} r_\pi^{-1}$, where $\mathcal{R}$ is defined in Lemma~\ref{ex:4} and $r_\pi$ is the rotation defined in Section~\ref{Pre}. 
Then, $H'=\H$ for $g\geq 2$, and $\mathcal{H}_1$ is generated by $t_{\alpha_1}$ and $t_{\alpha_2}$. 
\end{thm}
\begin{proof}
Suppose that $g\geq 2$. 
We set $f=t_{\alpha_1} t_{\alpha_2} t_{\alpha_3} \cdots t_{\alpha_{2g-1}}$. 
From Lemma~\ref{ex:4}, $t_{\alpha_{2g+1}}^{-1} = \mathcal{R}^{2g-1}$ is in $H'$, and therefore $(t_{\alpha_1} f)^{g-1} = \mathcal{R}t_{\alpha_{2g+1}}$ is also in $H'$. 
From these and the chain relation $(t_{\alpha_1}f)^{2g-1} = t_{\alpha_{2g+1}}^2$ (see the proof of Lemma~\ref{ex:4}), we see that $t_{\alpha_1}f = t_{\alpha_{2g+1}}^2 (t_{\alpha_1}f)^{-2(g-1)}$ is contained in $H'$. 
Therefore, by $(r_\pi \mathcal{R} r_\pi^{-1})^{2g-1} = r_\pi t_{\alpha_{2g+1}}^{-1} r_\pi^{-1} = t_{r_\pi(\alpha_{2g+1})}^{-1} = t_{\alpha_1}^{-1}$, the elements $t_{\alpha_1}$ and $f$ are in $H'$. 
Since it is easy to check that $f(\alpha_j) = \alpha_{j+1}$ for $j=1,2,\ldots,2g-2$, the elements $t_{\alpha_2}, t_{\alpha_3},\ldots,t_{\alpha_{2g-1}}$ are contained in $H'$. 
Here, we see that
\begin{align*}
r_\pi \mathcal{R} r_\pi^{-1} &= r_\pi (t_{\alpha_1}^2 t_{\alpha_2} t_{\alpha_3} \cdots t_{\alpha_{2g-1}})^{g-1} t_{\alpha_{2g+1}}^{-1} r_\pi^{-1} \\
&= (t_{r_\pi(\alpha_1)}^2 t_{r_\pi(\alpha_2)} t_{r_\pi(\alpha_3)} \cdots t_{r_\pi(\alpha_{2g-1})})^{g-1} t_{r_\pi(\alpha_{2g+1})}^{-1} \\
&= (t_{\alpha_{2g+1}}^2 t_{\alpha_{2g}} t_{\alpha_{2g-1}} \cdots t_{\alpha_3})^{g-1} t_{\alpha_1}^{-1}, 
\end{align*}
and hence $(t_{\alpha_{2g+1}}^2 t_{\alpha_{2g}} t_{\alpha_{2g-1}} \cdots t_{\alpha_3})^{g-1} = r_\pi \mathcal{R} r_\pi^{-1} t_{\alpha_1}$ is in $H'$ since $t_{\alpha_1}$ is in $H'$. 
When we set $h=t_{\alpha_{2g+1}}^2 t_{\alpha_{2g}} t_{\alpha_{2g-1}} \cdots t_{\alpha_3}$, $h^{g-1}$ is in $H'$. 
It is easily seen that $h^{-1} (\alpha_{i}) = \alpha_{i+1}$ for $i=3,4,\ldots,2g-1$, and hence $h^{-(g-1)}(\alpha_{g+1}) = \alpha_{2g}$. 
Therefore, $t_{\alpha_{2g}}$ is in $H'$ since $t_{\alpha_{g+1}}$ and $h^{g-1}$ are contained in $H'$. 
Summarizing, the elements $t_{\alpha_1},t_{\alpha_2},\ldots, t_{\alpha_{2g+1}}$ are contained in $H'$, and hence $H' = \H$ for $g\geq 2$. 
By Theorem~\ref{Lickorishthm}, $\mathcal{H}_1$ is generated by $t_{\alpha_1}$ and $t_{\alpha_2}$. 
This completes the proof. 
\end{proof}

\section{Proofs of Theorems~\ref{thm:3} and~\ref{thm:4}}
In this section, we prove Theorems~\ref{thm:3} and~\ref{thm:4} in the introduction. 
Theorem~\ref{thm:3} (resp.~\ref{thm:4}) follows from Theorems~\ref{thm:7} and~\ref{thm:10} (resp.~\ref{thm:8} and~\ref{thm:11}) below. 
\begin{thm}\label{thm:7}
Let $g\geq 9$, and let $G_1$ be the subgroup of $\M$ generated by $\rho_n$ and $(t_{b_3}t_{a_5}^{-1}t_{c_6}) \rho_n (t_{b_3}t_{a_5}^{-1}t_{c_6})^{-1}$, where $\rho_n$ is defined in Lemma~\ref{ex:2}. 
Then, $G_1 = \M$. 
\end{thm}

\begin{thm}\label{thm:10}
Let $g\geq 3$, and let $G_2$ be the subgroup of $\M$ generated by $\rho_n$, $t_{a_2} \rho_n t_{a_2}^{-1}$ and $(t_{a_3}t_{b_3}t_{c_2}) \rho_n (t_{a_3}t_{b_3}t_{c_2})^{-1}$, where $\rho_n$ is defined in Lemma~\ref{ex:2}. 
Then, $G_2 = \M$. 
\end{thm}

\begin{thm}\label{thm:8}
Let $g\geq 8$, and let $G_3$ be the subgroup of $\M$ generated by $\rho$ and $(t_{b_2}t_{a_4}^{-1}t_{c_5}) \rho (t_{b_2}t_{a_4}^{-1}t_{c_5})^{-1}$, where $\rho$ is defined in Lemma~\ref{ex:3}. 
Then, $G_3 = \M$. 
\end{thm}

\begin{thm}\label{thm:11}
Let $g\geq 3$, and let $G_4$ be the subgroup of $\M$ generated by $\rho'$, $t_{a_2} \rho' t_{a_2}^{-1}$ and $(t_{a_3}t_{b_3}t_{c_2}) \rho' (t_{a_3}t_{b_3}t_{c_2})^{-1}$, where $\rho'$ is defined in Lemma~\ref{ex:3}. 
Then, $G_4 = \M$. 
\end{thm}

Lemmas~\ref{lem:10}--\ref{lem:8} below are used to show Theorems~\ref{thm:7}--\ref{thm:11}. 
Recall that if a Dehn twist $t_c$ and an element $f$ are in a subgroup $G$ of $\M$, then $t_f(c)$ is also in $G$ by $t_{f(c)} = f t_c f^{-1}$. 
Therefore, if an element $h$ and a product $t_ct_d^{-1}$ of two Dehn twists $t_c$ and $t_d$ are in $G$, then $t_{h(c)}t_{h(d)}^{-1}$ is also in $G$. 
Let us consider the simple closed curves in Figure~\ref{rotation}. 
Note that $a_i$ (resp. $b_i$) and $b_i$ (resp. $c_i$) intersects transversely at exactly one point for $i=1,2,\ldots,g$, and that $b_i$ (resp. $b_1$) and $c_{i-1}$ (resp. $c_g$) intersects transversely at exactly one point for $i=2,3,\ldots,g$. 
All other pairs of curves are disjoint. 
We will repeatedly use these facts in the proofs of Lemmas~\ref{lem:10}--\ref{lem:8} and Theorems~\ref{thm:7}--\ref{thm:11}.  

\begin{lem}\label{lem:10}
Let $g\geq 3$, and let $G'$ be a subgroup of $\M$ which contains
\begin{align*}
t_{a_k} t_{a_{k+1}}^{-1}, \ t_{a_k} t_{a_{k+2}}^{-1}, \ t_{a_k} t_{b_{k+1}}^{-1}, \ t_{a_k} t_{c_k}^{-1}, \ t_{a_{k+1}} t_{c_{k+1}}^{-1}, \ t_{a_{k}} t_{c_{k+1}}^{-1}, \ t_{a_{k}} t_{b_{k+2}}^{-1}
\end{align*}
for some integer $k$ such that $1\leq k \leq g-2$. 
Then, the element $t_{a_{k+2}}$ is in $G'$.
\end{lem}
\begin{proof}
Let $d_1$ and $d_2$ be the simple closed curves in $\Sg$ as in Figure~\ref{curves1}. 
We set $\phi_1 = t_{b_{k+1}} t_{a_k}^{-1} \cdot  t_{c_k} t_{a_k}^{-1} \cdot t_{a_k} t_{a_{k+1}}^{-1} \cdot t_{c_{k+1}} t_{a_{k}}^{-1}$ and $\phi_2 = t_{b_{k+2}}t_{a_k}^{-1} \cdot t_{c_{k+1}} t_{a_k}^{-1} \cdot t_{a_{k+2}} t_{a_k}^{-1} \cdot t_{b_{k+2}} t_{a_k}^{-1}$. 
Then, it is easy to check that $\phi_1 (b_{k+1}, a_k) = (d_1, a_k)$ and $\phi_2 (d_1, a_k) = (d_2, a_k)$. 
Therefore, $t_{d_1} t_{a_k}^{-1} = (\phi_1 t_{a_k} t_{b_{k+1}}^{-1} \phi_1^{-1})^{-1}$ and $t_{d_2} t_{a_k}^{-1} =\phi_2 \cdot t_{d_1} t_{a_k}^{-1} \cdot \phi_2^{-1}$ are in $G'$ since $\phi_1, \phi_2$ and $t_{a_k} t_{b_{k+1}}^{-1}$ are in $G'$.  
Moreover, $t_{d_2} t_{c_k}^{-1} = t_{d_2} t_{a_k}^{-1} \cdot t_{a_k} t_{c_k}^{-1}$ is in $G'$ since $t_{a_k} t_{c_k}^{-1}$ is in $G'$. 
Therefore, by the lantern relation $t_{a_k} t_{c_k} t_{c_{k+1}} t_{a_{k+2}} = t_{a_{k+1}} t_{d_1} t_{d_2}$, we obtain that $t_{a_{k+2}} = t_{a_{k+1}} t_{c_{k+1}}^{-1} \cdot t_{d_1} t_{a_k}^{-1} \cdot t_{d_2} t_{c_k}^{-1}$ is in $G'$. 
\end{proof}
\begin{figure}[hbt]
  \centering
       \includegraphics[scale=.20]{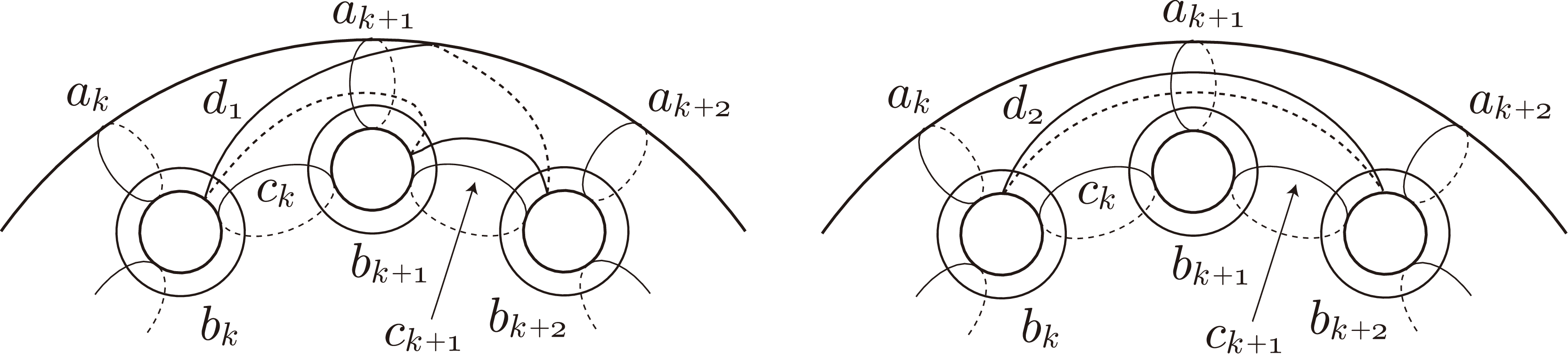}
       \caption{The simple closed curve $d_1$ and $d_2$ on $\Sigma_g$.}
       \label{curves1}
  \end{figure}

\begin{lem}\label{lem:7}
Let $g\geq k+2$, where $k$ is a positive integer, and let $f$ be an element in $\M$ satisfying $f (a_k, b_k, c_k) = (a_{k+1}, b_{k+1}, c_{k+1})$ and $f(a_{k+1}, b_{k+1}) = (a_{k+2}, b_{k+2})$. 
We denote by $G'$ the subgroup of $\M$ generated by $f$, $t_{a_k} t_{a_{k+1}}^{-1}$, $t_{b_k} t_{b_{k+1}}^{-1}$ and $t_{c_k} t_{c_{k+1}}^{-1}$. 
Then, the elements $t_{a_k}, t_{b_k}, t_{c_k}$ are in $G'$.
\end{lem}
\begin{proof}
Since $f, t_{a_k} t_{a_{k+1}}^{-1}$ and $t_{b_k} t_{b_{k+1}}^{-1}$ are in $G'$, we see that $t_{a_{k+1}} t_{a_{k+2}}^{-1}$ and $t_{b_{k+1}} t_{b_{k+2}}^{-1}$ are in $G'$. 
Note that it is easy to check that $t_{a_{k}} t_{a_{k+1}}^{-1} \cdot  t_{b_{k}} t_{b_{k+1}}^{-1} (a_{k+1},a_{k+2}) = (b_{k+1},a_{k+2})$. 
Since $t_{a_k} t_{a_{k+1}}^{-1}, t_{b_k} t_{b_{k+1}}^{-1}$ and $t_{a_{k+1}} t_{a_{k+2}}^{-1}$ are contained in $G'$, we see that $t_{b_{k+1}} t_{a_{k+2}}^{-1}$ and $t_{b_k} t_{a_{k+2}}^{-1} = t_{b_k} t_{b_{k+1}}^{-1} \cdot t_{b_{k+1}} t_{a_{k+2}}^{-1}$ are also contained in $G'$. 
Similarly, since $t_{b_k} t_{a_{k+2}}^{-1}$ and $t_{c_k} t_{c_{k+1}}^{-1}$ are in $G'$ and it is easily seen that $t_{b_k} t_{a_{k+2}}^{-1} \cdot t_{c_k} t_{c_{k+1}}^{-1} (b_k, a_{k+2}) = (c_k, a_{k+2})$, we obtain that $t_{c_{k}} t_{a_{k+2}}^{-1}$ is in $G'$. 
By the above argument, we showed that the following elements are in $G'$:
\begin{align*}
t_{a_k} t_{a_{k+1}}^{-1}, \ t_{a_{k+1}} t_{a_{k+2}}^{-1}, \ t_{b_{k+1}} t_{a_{k+2}}^{-1}, \ t_{c_{k}} t_{a_{k+2}}^{-1}. 
\end{align*}
Using these elements, we see that the elements 
\begin{align*}
t_{a_k} t_{a_{k+2}}^{-1} &= t_{a_k} t_{a_{k+1}}^{-1} \cdot t_{a_{k+1}} t_{a_{k+2}}^{-1}, \\
t_{a_k} t_{b_{k+1}}^{-1} &= t_{a_k} t_{a_{k+1}}^{-1} \cdot t_{a_{k+1}} t_{a_{k+2}}^{-1} \cdot (t_{b_{k+1}} t_{a_{k+2}}^{-1})^{-1}, \\
t_{a_k} t_{c_k}^{-1} &= t_{a_k} t_{a_{k+1}}^{-1} \cdot t_{a_{k+1}} t_{a_{k+2}}^{-1} \cdot (t_{c_k} t_{a_{k+2}}^{-1})^{-1}, \\
t_{a_{k+1}} t_{c_{k+1}}^{-1} &= t_{f(a_k)} t_{f(c_k)}^{-1} = f t_{a_k} t_{c_k}^{-1} f^{-1}, \\
t_{a_{k}} t_{c_{k+1}}^{-1} &= t_{a_k} t_{a_{k+1}}^{-1} \cdot t_{a_{k+1}} t_{c_{k+1}}^{-1}, \\ 
t_{a_{k+1}} t_{b_{k+2}}^{-1} &= t_{f(a_k)} t_{f(b_{k+1})}^{-1} = f t_{a_k} t_{b_{k+1}} f^{-1}, \ \ \mathrm{and} \\
t_{a_{k}} t_{b_{k+2}}^{-1} &= t_{a_k} t_{a_{k+1}}^{-1} \cdot t_{a_{k+1}} t_{b_{k+2}}^{-1}
\end{align*}
are contained in $G'$. 
As a consequence, we verified that $G'$ satisfies the condition on Lemma ~\ref{lem:10}, that is  
\begin{align*}
t_{a_k} t_{a_{k+1}}^{-1}, \ t_{a_k} t_{a_{k+2}}^{-1}, \ t_{a_k} t_{b_{k+1}}^{-1}, \ t_{a_k} t_{c_k}^{-1}, \ t_{a_{k+1}} t_{c_{k+1}}^{-1}, \ t_{a_{k}} t_{c_{k+1}}^{-1}, \ t_{a_{k}} t_{b_{k+2}}^{-1}, 
\end{align*} 
are elements of $G'$. 
Then we see that $t_{a_{k+2}}$ is in $G'$ by applying Lemma~\ref{lem:10}. 
Moreover,  the element $t_{a_k}$ is contained in $G'$ by the assumption that $f^2(a_k) = a_{k+2}$. 
Therefore, $t_{b_{k+1}}$ and $t_{c_{k+1}}$ are contained in $G'$ since $G'$ contains $t_{a_k} t_{b_{k+1}}^{-1}$ and $t_{a_{k}} t_{c_{k+1}}^{-1}$. 
Similarly, by the assumption that $f^{-1}(b_{k+1}, c_{k+1}) = (b_k,c_k)$, $t_{b_k}$ and $t_{c_k}$ are contained in $G'$ since $t_{b_{k+1}}$, $t_{c_{k+1}}$ are in $G'$, and the lemma follows. 
\end{proof}

\begin{lem}\label{lem:8}
Let $g\geq k+6$, where $k$ is a positive integer, and let $f$ be an element in $\M$ satisfying $f^j (a_k, b_k, c_k) = (a_{k+j}, b_{k+j}, c_{k+j})$ for $j=1,2,3,4$ and $f(c_{k+4}) = c_{k+5}$. 
We denote by $G''$ the subgroup of $\M$ generated by $f$ and $t_{b_k} t_{a_{k+2}}^{-1} t_{c_{k+3}} \cdot t_{c_{k+4}}^{-1} t_{a_{k+3}} t_{b_{k+1}}^{-1}$. 
Then, the elements $t_{a_k} t_{a_{k+1}}^{-1}, t_{b_k} t_{b_{k+1}}^{-1}, t_{c_k} t_{c_{k+1}}^{-1}$ are contained in $G''$. 
\end{lem}
\begin{proof}
We set 
\begin{align*}
\Phi_k &= t_{b_k} t_{a_{k+2}}^{-1} t_{c_{k+3}} \cdot t_{c_{k+4}}^{-1} t_{a_{k+3}} t_{b_{k+1}}^{-1}, \ \ \mathrm{and} \\
\Phi_{k+1} &= t_{b_{k+1}} t_{a_{k+3}}^{-1} t_{c_{k+4}} \cdot t_{c_{k+5}}^{-1} t_{a_{k+4}} t_{b_{k+2}}^{-1}. 
\end{align*} 
Note that $\Phi_k$ is in $G''$. Then, $\Phi_{k+1} = f \Phi_k f^{-1}$ is also $G''$ by the assumption on $f$. 
Since it is easy to check that 
\[\Phi_{k+1} \Phi_k  (b_{k+1},  a_{k+3},  c_{k+4},  c_{k+5},  a_{k+4},  b_{k+2})  =  (b_{k+1},  a_{k+3},  c_{k+4},  c_{k+5},  a_{k+4},  a_{k+2}),\]
we see that the element 
\begin{align*}
t_{b_{k+1}} t_{a_{k+3}}^{-1} t_{c_{k+4}} \cdot t_{c_{k+5}}^{-1} t_{a_{k+4}} t_{a_{k+2}}^{-1} = (\Phi_{k+1} \Phi_k) \Phi_{k+1} (\Phi_{k+1} \Phi_k)^{-1} 
\end{align*}
is contained in $G''$. 
Then, we obtain the following:
\begin{align}
t_{a_{k+2}} t_{b_{k+2}}^{-1} = (t_{b_{k+1}} t_{a_{k+3}}^{-1} t_{c_{k+4}} \cdot t_{c_{k+5}}^{-1} t_{a_{k+4}} t_{a_{k+2}}^{-1})^{-1} \Phi_{k+1} \in G''. \label{eq:0}
\end{align}
Moreover, from the assumption on $f$, the following holds: 
\begin{align}
t_{a_{k+3}} t_{b_{k+3}}^{-1} \in G'' \ and \label{eq:1} \\
t_{a_{k+4}} t_{b_{k+4}}^{-1} \in G''. \label{eq:2}
\end{align}
Here, we find that $t_{b_{k+2}}^{-1} t_{a_{k+2}} =  \Phi_{k+1} (t_{b_{k+1}} t_{a_{k+3}}^{-1} t_{c_{k+4}} \cdot t_{c_{k+5}}^{-1} t_{a_{k+4}} t_{a_{k+2}}^{-1})^{-1}$, and hence $t_{b_{k+2}}^{-1} t_{a_{k+2}}$ is in $G''$. 
By similar arguments to \eqref{eq:1} and \eqref{eq:2}, we obtain the following:
\begin{align}
&t_{b_{k+3}}^{-1} t_{a_{k+3}}  \in G'', \ \ \mathrm{and} \label{eq:3} \\
&t_{b_{k+4}}^{-1} t_{a_{k+4}}  \in G''. \label{eq:4}
\end{align}
By \eqref{eq:3}, we see that $\Phi_k t_{a_{k+3}}^{-1} t_{b_{k+3}} = \Phi_k (t_{b_{k+3}}^{-1} t_{a_{k+3}})^{-1}$ is in $G''$. 
Moreover, we have
\begin{align*}
\Phi_k t_{a_{k+3}}^{-1} t_{b_{k+3}} 
&=  t_{b_k} t_{a_{k+2}}^{-1} t_{c_{k+3}} \cdot t_{c_{k+4}}^{-1} t_{a_{k+3}} t_{b_{k+1}}^{-1} \cdot t_{a_{k+3}}^{-1} t_{b_{k+3}} \\
&=  t_{b_k} t_{a_{k+2}}^{-1} t_{c_{k+3}} \cdot t_{c_{k+4}}^{-1} t_{b_{k+3}} t_{b_{k+1}}^{-1}. 
\end{align*}
This gives $(\Phi_k t_{a_{k+3}}^{-1} t_{b_{k+3}})^{-1}(a_{k+3}, b_{k+3}) = (t_{b_{k+3}}^{-1}(a_{k+3}), c_{k+3})$. 
Therefore, since $t_{a_{k+3}} t_{b_{k+3}}^{-1}$ and $(\Phi_k t_{a_{k+3}}^{-1} t_{b_{k+3}})^{-1}$ are in $G''$, the element $t_{t_{b_{k+3}}^{-1}(a_{k+3})} t_{c_{k+3}}^{-1}$ is in $G''$. 
That is, $t_{b_{k+3}}^{-1} t_{a_{k+3}} t_{b_{k+3}} t_{c_{k+3}}^{-1}$ is in $G''$. 
By \eqref{eq:3}, we have the following:
\begin{align}
t_{b_{k+3}} t_{c_{k+3}}^{-1} \in G''. \label{eq:5}
\end{align}
Since $\Phi_k$ and $\Phi_{k+1}$ are contained in $G''$, by \eqref{eq:4} we see that $\Phi_k \Phi_{k+1} t_{a_{k+4}}^{-1} t_{b_{k+4}} = \Phi_k \Phi_{k+1} (t_{b_{k+4}}^{-1} t_{a_{k+4}})^{-1}$ is in $G''$. 
Therefore, we obtain that 
\begin{align*}
\Phi_k \Phi_{k+1} t_{a_{k+4}}^{-1} t_{b_{k+4}} 
&=  t_{b_k} t_{a_{k+2}}^{-1} t_{c_{k+3}} \cdot t_{c_{k+5}}^{-1} t_{a_{k+4}} t_{b_{k+2}}^{-1} \cdot t_{a_{k+4}}^{-1} t_{b_{k+4}} \\
&=  t_{b_k} t_{a_{k+2}}^{-1} t_{c_{k+3}} \cdot t_{c_{k+5}}^{-1} t_{b_{k+4}} t_{b_{k+2}}^{-1}
\end{align*}
is in $G''$. 
This gives $(\Phi_k \Phi_{k+1} t_{a_{k+4}}^{-1} t_{b_{k+4}})^{-1}(a_{k+4}, b_{k+4}) =  (t_{b_{k+4}}^{-1}(a_{k+4}), c_{k+3})$, and hence we see that $t_{t_{b_{k+4}}^{-1}(a_{k+4})} t_{c_{k+3}}^{-1}$ is in $G''$ by \eqref{eq:2}. 
That is, $t_{b_{k+4}}^{-1} t_{a_{k+4}} t_{b_{k+4}} t_{c_{k+3}}^{-1}$ is in $G''$.  
Therefore, by \eqref{eq:4} and the assumption that $f^{-1}(b_{k+4}, c_{k+3}) = (b_{k+3}, c_{k+2})$, the following holds: 
\begin{align}
&t_{b_{k+4}} t_{c_{k+3}}^{-1} \in G'', \ \ \mathrm{and} \label{eq:6} \\
&t_{b_{k+3}} t_{c_{k+2}}^{-1} \in G''. \label{eq:7} 
\end{align}
Here, by \eqref{eq:5} and \eqref{eq:6}, the element $t_{b_{k+3}} t_{b_{k+4}}^{-1} = t_{b_{k+3}} t_{c_{k+3}}^{-1} (t_{b_{k+4}} t_{c_{k+3}}^{-1})^{-1}$ is  in $G''$. 
Therefore, by the assumption that $f^{-l}(b_{k+3}, b_{k+4}) = (b_{k+3-l}, b_{k+4-l})$ for $l=1,2,3$, 
$t_{b_{k}} t_{b_{k+1}}^{-1}$, $t_{b_{k+1}} t_{b_{k+2}}^{-1}$ and $t_{b_{k+2}} t_{b_{k+3}}^{-1}$ are contained in $G''$.
Moreover, 
by \eqref{eq:0}, \eqref{eq:1}, we obtain that $t_{a_{k+2}} t_{a_{k+3}}^{-1} = t_{a_{k+2}} t_{b_{k+2}}^{-1} \cdot t_{b_{k+2}} t_{b_{k+3}}^{-1} \cdot (t_{a_{k+3}} t_{b_{k+3}}^{-1})^{-1}$ is in $G''$. 
By the assumption $f^{-2}(a_{k+2}, a_{k+3}) = (a_k, a_{k+1})$, we obtain that $t_{a_k} t_{a_{k+1}}^{-1} \in G''$. 
By \eqref{eq:7}, \eqref{eq:5}, the element $t_{c_{k+2}} t_{c_{k+3}}^{-1} = (t_{b_{k+3}} t_{c_{k+2}}^{-1})^{-1} t_{b_{k+3}} t_{c_{k+3}}^{-1}$ is in $G''$, and therefore $t_{c_k} t_{c_{k+1}}^{-1}$ is in $G''$ by the assumption that $f^{-2}(c_{k+2}, c_{k+3}) = (c_k, c_{k+1})$. 
Summarizing the above, the elements $t_{a_k} t_{a_{k+1}}^{-1}, t_{b_k} t_{b_{k+1}}^{-1}, t_{c_k}t_{c_{k+1}}$ are contained in $G''$, and the lemma follows. 
\end{proof}

\begin{proof}[Proof of Theorem~\ref{thm:7}]
It is easy to check that 
\begin{align*}
\rho_n(b_3,a_5,c_6) = r t_{c_1} t_{b_1}^{-n} t_{a_1}(b_3,a_5,c_6) = r(b_3,a_5,c_6) = (b_4, a_6, c_7).
\end{align*} 
This gives 
\begin{align*}
(t_{b_3}t_{a_5}^{-1}t_{c_6}) \rho_n (t_{b_3}t_{a_5}^{-1}t_{c_6})^{-1}  &= (t_{b_3}t_{a_5}^{-1}t_{c_6}) \cdot \rho_n (t_{b_3}t_{a_5}^{-1}t_{c_6})^{-1} \rho_n^{-1}\cdot \rho_n \\
&= (t_{b_3}t_{a_5}^{-1}t_{c_6}) \cdot (t_{b_4}t_{a_6}^{-1}t_{c_7})^{-1} \rho_n. 
\end{align*}
and therefore, 
$G_1$ contains the element $(t_{b_3}t_{a_5}^{-1}t_{c_6}) (t_{b_4}t_{a_6}^{-1}t_{c_7})^{-1}$. 
From Lemma~\ref{lem:8}, the elements $t_{a_3}t_{a_4}^{-1}$, $t_{b_3}t_{b_4}^{-1}$ and $t_{c_3}t_{c_4}^{-1}$ are in $G_1$ (by letting $k=3$ and $f=\rho_n$). 
Moreover, from Lemma~\ref{lem:7}, we see that the elements $t_{a_3}, t_{b_3}, t_{c_3}$ are in $G_1$ (by letting $k=3$ and $f=\rho_n$). 
Since it is easy to check that $\rho_n (a_i,b_i) = (a_{i+1}, b_{i+1})$ for $i=3,4,\ldots,g-1$ and $\rho_n(a_g, b_g) = (a_1, b_1)$, we see that $t_{a_1}$ and $t_{b_1}$ are contained in $G_1$. 
Hence, the element $r t_{c_1} = r t_{c_1} t_{b_1}^{-n} t_{a_1} \cdot t_{a_1}^{-1} t_{b_1}^n = \rho_n t_{a_1}^{-1} t_{b_1}^n$ is in $G_1$. 
By $(r t_{c_1})^{-2}(c_3) = c_1$, we obtain that $t_{c_1} = t_{(r t_{c_1})^{-2}(c_3)} = (r t_{c_1})^{-2} t_{c_3} (r t_{c_1})^{2}$ is in $G_1$, and therefore, $r = r t_{c_1} \cdot t_{c_1}^{-1}$ is in $G_1$. 
Since $r^i(a_1,b_1,c_1) = (a_{i+1},b_{i+1},c_{i+1})$ for $i=0,1,2,\ldots,g-1$, the elements $t_{a_{i+1}}$, $t_{b_{i+1}}$ and $t_{c_{i+1}}$ are in $G_1$ 
for $i=0,1,2,\ldots,g-1$. 
Therefore, we see that $G_1 = \M$, which proves the theorem. 
\end{proof}

\begin{proof}[Proof of Theorem~\ref{thm:8}]
We omit the proof since it is almost the same as the proof of Theorem~\ref{thm:7}.  
\end{proof}

\begin{proof}[Proof of Theorem~\ref{thm:10}]
By the definition of $\rho_n$, we obtain 
\begin{align*}
\rho_n(a_2) &= a_3, \ \ \mathrm{and} \\ 
\rho_n(a_3,b_3,c_2) &=
\begin{cases}
(a_4, b_4, c_3) & (g\geq 4), \\
(a_1, b_1, c_3) & (g = 3). 
\end{cases}
\end{align*}
Since we have 
\begin{align*}
t_{a_2} \rho_n t_{a_2}^{-1} &= t_{a_2} (\rho_n t_{a_2}^{-1} \rho_n^{-1}) \rho_n \\
(t_{a_3}t_{b_3}t_{c_2}) \rho_n (t_{a_3}t_{b_3}t_{c_2})^{-1} &= (t_{a_3}t_{b_3}t_{c_2}) \cdot \rho_n (t_{a_3}t_{b_3}t_{c_2})^{-1} \rho_n^{-1} \cdot \rho_n, 
\end{align*}
the following holds:
\begin{align*}
&G_2 \ni t_{a_2} \rho_n t_{a_2}^{-1} = t_{a_2} t_{a_3}^{-1} \rho_n, \ \ \mathrm{and} \\
&G_2 \ni (t_{a_3}t_{b_3}t_{c_2}) \rho_n (t_{a_3}t_{b_3}t_{c_2})^{-1} = 
\begin{cases}
(t_{a_3}t_{b_3}t_{c_2}) \left(t_{a_4}t_{b_4}t_{c_3}\right)^{-1} \rho_n & (g\geq 4), \\
(t_{a_3}t_{b_3}t_{c_2}) \left(t_{a_1}t_{b_1}t_{c_3}\right)^{-1} \rho_n & (g=3). 
\end{cases}
\end{align*}
Therefore, we obtain the following: 
\begin{align}
&t_{a_2} t_{a_3}^{-1} \in G_2, \label{ele:1} \\ 
&t_{a_3}t_{b_3}t_{c_2} \cdot t_{c_3}^{-1} t_{b_4}^{-1} t_{a_4}^{-1} \in G_2 \ \mathrm{if} \ g\geq 4, \ \ \mathrm{and} \label{ele:111} \\
&t_{a_3}t_{b_3}t_{c_2} \cdot t_{c_3}^{-1} t_{b_1}^{-1} t_{a_1}^{-1} \in G_2 \ \mathrm{if} \ g=3. \label{ele:112}
\end{align}

Hereafter, we prove that $t_{a_3}$ is in $G_2$ using these elements and Lemma~\ref{lem:10}. 
More precisely, we show that the following elements are contained in $G_2$: 
\begin{align*}
t_{a_1} t_{a_2}^{-1}, \ t_{a_1} t_{a_3}^{-1}, \ t_{a_1} t_{b_2}^{-1}, \ t_{a_1} t_{c_1}^{-1}, \ t_{a_2} t_{c_2}^{-1}, \ t_{a_1} t_{c_2}^{-1}, \ t_{a_1} t_{b_3}^{-1}. 
\end{align*}

Suppose that $g\geq 4$. 
For simplicity of notation, we write $\phi = t_{a_3}t_{b_3}t_{c_2} \cdot t_{c_3}^{-1} t_{b_4}^{-1} t_{a_4}^{-1}$.

We first show that 
$t_{a_1} t_{a_2}^{-1}$, $t_{a_1} t_{a_3}^{-1}$ and $t_{a_1} t_{b_3}^{-1}$ are in $G_2$.  
It is easy to check that $\phi(a_2,a_3) = (a_2,b_3)$, and hence 
$t_{a_2} t_{b_3}^{-1}$ is in $G_2$ by \eqref{ele:1} and \eqref{ele:111}. 
We obtain the following: 
\begin{align}
&t_{a_{2+i}} t_{a_{3+i}}^{-1} = \rho_n^i t_{a_2} t_{a_3}^{-1} \rho_n^{-i} \in G_2 \ \ (0\leq i \leq g-3), \\
&t_{a_g} t_{a_1}^{-1} = \rho_n^{g-2} t_{a_2} t_{a_3}^{-1} \rho_n^{-(g-2)} \in G_2 , \label{A1} \\
&t_{a_g} t_{b_1}^{-1} = \rho_n^{g-2} t_{a_2} t_{b_3}^{-1} \rho_n^{-(g-2)}\in G_2 . \label{A2}
\end{align}
Using these, we see that the element $t_{a_2} t_{a_1}^{-1} = t_{a_2} t_{a_3}^{-1} \cdot t_{a_3} t_{a_4}^{-1} \cdots t_{a_{g-1}} t_{a_g}^{-1} \cdot t_{a_g} t_{a_1}^{-1}$ is in $G_2$. 
Therefore, the following holds: 
\begin{align}
&t_{a_1} t_{a_2}^{-1} = t_{a_2} t_{a_1}^{-1} \in G_2, \label{A3}\\
&t_{a_1} t_{a_3}^{-1} = t_{a_1} t_{a_2}^{-1} \cdot t_{a_2} t_{a_3}^{-1} \in G_2, \label{A4} \\
&t_{a_1}^{-1} t_{b_1} = (t_{a_g} t_{a_1}^{-1}) \cdot (t_{a_g} t_{b_1}^{-1})^{-1} \in G_2, \ \ \mathrm{and} \label{A5} \\
&t_{a_1} t_{b_1}^{-1} = (t_{a_g} t_{a_1}^{-1})^{-1} (t_{a_g} t_{b_1}^{-1}) \in G_2. \notag
\end{align}
Moreover, since $\rho_n^{-g+2}(a_1,b_1) = (a_3,b_3)$, we have the following:
\begin{align}
&t_{a_3} t_{b_3}^{-1} \in G_2, \ \ \mathrm{and} \label{ele:2} \\
&t_{a_1} t_{b_3}^{-1} = t_{a_1} t_{a_3}^{-1} \cdot t_{a_3} t_{b_3}^{-1} \in G_2. \label{A6}
\end{align}
From \eqref{A3}, \eqref{A4} and \eqref{A6}, the elements $t_{a_1} t_{a_2}^{-1}$, $t_{a_1} t_{a_3}^{-1}$ and $t_{a_1} t_{b_3}^{-1}$ are in $G_2$.

Next, we show that $t_{a_1} t_{b_2}^{-1}$, $t_{a_1} t_{c_1}^{-1}$, $t_{a_2} t_{c_2}^{-1}$ and $t_{a_1} t_{c_2}^{-1}$ are in $G_2$. 
By \eqref{ele:1} and \eqref{ele:2}, we have $t_{a_2} t_{b_3}^{-1} = t_{a_2} t_{a_3}^{-1} \cdot t_{a_3} t_{b_3}^{-1}$ is in $G_2$. 
Therefore, by $\rho_n(a_2,b_3) = (a_3,b_4)$, the element $t_{a_3}t_{b_4}^{-1}$ is contained in $G_2$. 
Combining this and \eqref{A4}, the element $t_{a_1} t_{b_4}^{-1} = t_{a_1} t_{a_3}^{-1} \cdot t_{a_3} t_{b_4}^{-1}$ is in $G_2$. 
Then, since it is easily seen that $\phi^{-1}(a_1,b_4) = (a_1,c_3)$, we see that $t_{a_1} t_{c_3}^{-1}$ is in $G_2$ by \eqref{ele:111}. 
Using this, \eqref{A3} and \eqref{A4}, the elements $t_{a_2} t_{c_3}^{-1} = (t_{a_1} t_{a_2}^{-1})^{-1} t_{a_1} t_{c_3}^{-1}$ and $t_{a_3} t_{c_3}^{-1} = (t_{a_1} t_{a_3}^{-1})^{-1} t_{a_1} t_{c_3}^{-1}$ are contained in $G_2$. 
Here, we set $\sigma_n = \rho_n \cdot t_{a_1}^{-1} t_{b_1} \cdot (t_{a_2}^{-1} t_{a_1} \cdot t_{a_1}^{-1} t_{b_1})^{n-1}$, which is in $G_2$ by \eqref{A3} and \eqref{A5}. 
Then, by the definition of $\rho_n$, we have $\sigma_n = r t_{c_1} t_{a_2}^{-n+1}$. 
This gives $\sigma_n^{-2}(a_3,c_3) = (a_1,c_1)$, $\sigma_n^{-1}(a_3,c_3) = (a_2,c_2)$ and $\sigma_n^{-2}(a_2,c_3) = (a_1,c_2)$. 
Therefore, the following holds: 
\begin{align}
&t_{a_1} t_{c_1}^{-1} \in G_2, \label{A7} \\
&t_{a_2} t_{c_2}^{-1} \in G_2, \ \ \mathrm{and} \label{A8} \\
&t_{a_1} t_{c_2}^{-1} \in G_2. \label{A9}
\end{align}
Moreover, since it is immediate that $\sigma_n(a_1,b_1) = (a_2,t_{c_2}(b_2))$, the element $t_{a_2}^{-1} t_{t_{c_2}(b_2)}$ is in $G_2$ by \eqref{A5}. 
By \eqref{A8} and $t_{a_2}^{-1} t_{t_{c_2}(b_2)} = t_{a_2}^{-1} t_{c_2} t_{b_2} t_{c_2}^{-1}$, we see that $t_{c_2} t_{b_2}^{-1}$ is in $G_2$. 
Therefore, the element $t_{a_1} t_{b_2}^{-1} = t_{a_1} t_{c_2}^{-1} \cdot t_{c_2} t_{b_2}^{-1}$ is contained in $G_2$ by \eqref{A9}. 
Summarizing the above, the elements $t_{a_1} t_{b_2}^{-1}$, $t_{a_1} t_{c_1}^{-1}$, $t_{a_2} t_{c_2}^{-1}$ and $t_{a_1} t_{c_2}^{-1}$ are in $G_2$.

By Lemma~\ref{lem:10} we see that $t_{a_3}$ is in $G_2$. 
Moreover, by \eqref{A4}, \eqref{A5} and \eqref{A7}, the Dehn twists $t_{a_1} = t_{a_1}t_{a_3}^{-1} \cdot t_{a_3}$, $t_{b_1} = t_{a_1} \cdot t_{a_1}^{-1} t_{b_1}$ and $t_{c_1} = (t_{a_1}t_{c_1}^{-1})^{-1} \cdot t_{a_1}^{-1}$ are contained in $G_2$, and hence the rotation $r = \rho_n t_{a_1}^{-1} t_{b_1}^n t_{c_1}^{-1}$ is in $G_2$. 
Since $r^i(a_1,b_1,c_1) = (a_{i+1},b_{i+1},c_{i+1})$ for $i=0,1,2,\ldots,g-1$, we obtain that the Dehn twists $t_{a_{i+1}}$, $t_{b_{i+1}}$ and $t_{c_{i+1}}$ are in $G_2$ for $i=0,1,2,\ldots,g-1$. 
Therefore, we see that $G_2 = \M$ for $g\geq 4$.

Suppose that $g=3$. 
Since it is easy to check that $\rho_n(a_2,a_3) = (a_3,a_1)$, by \eqref{ele:1} the following holds: 
\begin{align}
&t_{a_3} t_{a_1}^{-1} = t_{a_1}^{-1} t_{a_3} \in G_2, \ \ \mathrm{and} \label{B1} \\
&t_{a_1} t_{a_2}^{-1} =  (t_{a_3} t_{a_1}^{-1})^{-1} \cdot (t_{a_2} t_{a_3}^{-1})^{-1} \in G_2. \label{B5}
\end{align}
Here, for simplicity of notation, we also write $\phi = t_{a_3}t_{b_3}t_{c_2} \cdot t_{c_3}^{-1} t_{b_1}^{-1} t_{a_1}^{-1}$. 
Then, by \eqref{ele:1}, \eqref{ele:112} and $\phi(a_2,a_3) = (a_2,b_3)$, we see that $t_{a_2} t_{b_3}^{-1}$ is in $G_2$. 
From this, \eqref{B5} and $\rho_n(a_2,b_3) = (a_3,b_1)$, we have the following: 
\begin{align}
&t_{a_1} t_{b_3}^{-1} = t_{a_1} t_{a_2}^{-1} \cdot t_{a_2} t_{b_3}^{-1} \in G_2, \ \ \mathrm{and} \label{B6} \\
&t_{a_3} t_{b_1}^{-1} \in G_2. \label{B2}
\end{align}
Moreover, by \eqref{ele:1}, \eqref{B1} and \eqref{B2}, we obtain the following: 
\begin{align}
&t_{a_1}^{-1} t_{b_1} = t_{a_3} t_{a_1}^{-1} \cdot (t_{a_3} t_{b_1}^{-1})^{-1} \in G_2, \ \ \mathrm{and} \label{B7} \\
&t_{a_2} t_{b_1}^{-1} = t_{a_2} t_{a_3}^{-1} \cdot t_{a_3} t_{b_1}^{-1} \in G_2. \label{B8}
\end{align}
Here, when we set $\sigma_n=\rho_n \cdot t_{a_1}^{-1}t_{b_1} (t_{a_2}t_{b_1}^{-1})^{-n+1}$, which is in $G_2$ by \eqref{B7} and \eqref{B8}, we see that $\sigma_n = r t_{c_1} t_{a_2}^{-n+1}$ by the definition of $\rho_n$. 
Since it is easy to check that $\phi^{-1}(a_2,b_1) = (a_2,c_3)$, the element $t_{a_2} t_{c_3}^{-1}$ is in $G_2$ by \eqref{B8} and \eqref{ele:112}. 
Moreover, by $\sigma_n^2(a_2,c_3) = \sigma_n(a_3,c_1) = (a_1,c_2)$, the follwing holds: 
\begin{align}
&t_{a_3} t_{c_1}^{-1} \in G_2, \ \ \mathrm{and} \notag \\ 
&t_{a_1} t_{c_2}^{-1} \in G_2. \label{B9}
\end{align}
From these, \eqref{B1} and \eqref{B5}, we have
\begin{align}
&t_{a_1} t_{c_1}^{-1} = t_{a_1} t_{a_3}^{-1} \cdot t_{a_3} t_{c_1}^{-1} \in G_2, \ \ \mathrm{and} \label{B10} \\
&t_{c_2}^{-1} t_{a_2} = t_{a_2} t_{c_2}^{-1} = (t_{a_1} t_{a_2}^{-1})^{-1} \cdot t_{a_1} t_{c_2}^{-1} \in G_2 \label{B11}. 
\end{align}
By $\sigma_n(a_1,b_1) = (a_2, t_{c_2}(b_2))$ and \eqref{B7}, the element $t_{a_2}^{-1} t_{t_{c_2}(b_2)} = t_{a_2}^{-1} t_{c_2} t_{b_2} t_{c_2}^{-1}$ is in $G_2$. 
Therefore, the element $t_{b_2} t_{c_2}^{-1}$ is in $G_2$ by \eqref{B11}, and the element $t_{a_1} t_{b_2}^{-1} = t_{a_1} t_{c_2}^{-1} \cdot (t_{b_2} t_{c_2}^{-1})^{-1}$ is in $G_2$ by \eqref{B9}. 
From \eqref{B1}--\eqref{B6}, \eqref{B9}--\eqref{B11} and this, the Dehn twist $t_{a_3}$ is in $G_2$ using Lemma~\ref{lem:10}. 
By a similar argument to that in the case of $g\geq 4$, we see that $G_2 = \mathcal{M}_3$. 
\end{proof}

\begin{proof}[Proof of Theorem~\ref{thm:11}]
Suppose that $g\geq 4$. 
By a similar argument to the proof of Theorem~\ref{thm:10}, we see that $t_{a_1} t_{a_2}^{-1}$, $t_{a_1} t_{a_3}^{-1}$ and $t_{a_1} t_{b_3}^{-1}$ are in $G_4$. 
Moreover, the element $\rho' = r t_{c_1}$ is equal to $\sigma_1$ in the proof of Theorem~\ref{thm:10}. 
Therefore, $t_{a_1} t_{b_2}^{-1}$, $t_{a_1} t_{c_1}^{-1}$, $t_{a_2} t_{c_2}^{-1}$ and $t_{a_1} t_{c_2}^{-1}$ are in $G_4$, and hence the elements $t_{a_1}$, $t_{b_1}$, $t_{c_1}$ and $r$ are in $G_4$. 
By a similar argument to the proof of Theorem~\ref{thm:7}, we see that $G_4 = \M$ for $g\geq 4$. 
The same reasoning applies to the case $g=3$. 
This finishes the proof. 
\end{proof}

\end{document}